\documentclass[11pt,reqno]{amsart}

\usepackage[T1]{fontenc}
\usepackage[utf8]{inputenc}
\usepackage{lmodern}
\usepackage{microtype}
\usepackage{mathtools}
\usepackage{amssymb,amsfonts}
\usepackage{amscd}
\usepackage{enumitem}
\usepackage{xcolor}
\usepackage{booktabs}
\usepackage{hyperref}
\usepackage[nameinlink,capitalise,noabbrev]{cleveref}

\hypersetup{
  colorlinks=true,
  linkcolor=blue!55!black,
  citecolor=green!45!black,
  urlcolor=blue!65!black,
  pdftitle={On Cap-Decorated Floer Persistence},
  pdfauthor={Wenmin Gong}
}

\setlist[itemize]{leftmargin=2em,itemsep=2pt,topsep=4pt}
\setlist[enumerate]{leftmargin=2.2em,itemsep=2pt,topsep=4pt}
\newtheorem{theorem}{Theorem}[section]
\newtheorem{proposition}[theorem]{Proposition}
\newtheorem{lemma}[theorem]{Lemma}
\newtheorem{corollary}[theorem]{Corollary}
\theoremstyle{definition}
\newtheorem{definition}[theorem]{Definition}
\newtheorem{example}[theorem]{Example}
\theoremstyle{remark}
\newtheorem{remark}[theorem]{Remark}

\newcommand{\K}{\mathbb K}
\newcommand{\R}{\mathbb R}
\newcommand{\N}{\mathbb N}
\newcommand{\Z}{\mathbb Z}
\newcommand{\Spec}{\operatorname{Spec}}
\newcommand{\Ham}{\operatorname{Ham}}
\newcommand{\Powers}{\operatorname{Powers}}
\newcommand{\im}{\operatorname{im}}
\newcommand{\id}{\operatorname{id}}
\newcommand{\cl}{\operatorname{cl}}
\newcommand{\mcl}{\operatorname{mcl}}
\newcommand{\PCL}{\operatorname{PCL}}
\newcommand{\depth}{\operatorname{depth}}
\newcommand{\distE}{d_{\mathrm E}}

\newcommand{\distH}{d_{\mathrm H}}

\newcommand{\barc}{\operatorname{barc}}
\newcommand{\dgm}{\operatorname{dgm}}
\newcommand{\cA}{\mathcal A}

\newcommand{\cV}{\mathcal V}
\newcommand{\cW}{\mathcal W}
\newcommand{\eps}{\varepsilon}
\newcommand{\frownH}{\mathbin{\frown_H}}
\newcommand{\smilecup}{\mathbin{\smile}}
\newcommand{\normHofer}[1]{\lVert #1\rVert}
\newcommand{\mideal}{\mathfrak m}
\newcommand{\CapLayer}{\operatorname{Cap}}
\newcommand{\Bcap}{\mathcal B^{\mathrm{cap}}}
\newcommand{\Dcap}{\mathfrak D_{\mathrm{cap}}}
\newcommand{\PD}{\operatorname{PD}}
\newcommand{\supp}{\operatorname{supp}}
\newcommand{\Symp}{\operatorname{Symp}}
\newcommand{\OrbHam}{\mathcal O_{\mathrm{Ham}}}
\newcommand{\dHamConj}{d_{\mathrm H}^{\mathrm{Ham\text{-}conj}}}
\newcommand{\Csymp}[2]{\mathcal C_{#1}(#2)}
\newcommand{\DeltaCap}{\Delta^{\mathrm{cap}}}
\newcommand{\QHam}{\mathcal Q_{\mathrm{Ham}}}
\newcommand{\ProfileCap}{\mathfrak P_{\mathrm{cap}}}

\title[On Cap-Decorated Floer Persistence Modules]{On Cap-Decorated Floer Persistence Modules}
\author{Wenmin Gong}
\date{}

\address{School of Mathematical Sciences, Beijing Normal University,
	Beijing, 100875, China}

\email{ wmgong@bnu.edu.cn}

\begin{document}

\begin{abstract}
For a closed symplectically aspherical manifold, we introduce
cap-decorated Floer persistence modules by refining the action-filtered Hamiltonian
Floer persistence module with layers defined by iterated actions of homogeneous
cohomology ideals. We prove functoriality and stability of these layers, and
construct a universal cap-profile pseudodistance bounded above by Hofer
distance. As an application, we obtain new invariants of Hamiltonian
conjugacy classes, detecting symplectic mapping-class displacement phenomena invisible to
ordinary Floer barcodes and spectral invariants. In particular, for a
genus-two surface we construct Hamiltonian diffeomorphisms with identical
ordinary Floer data but positive separation in the Hamiltonian-conjugacy
quotient. 
\end{abstract}

\maketitle
\tableofcontents

\section{Introduction}

Let $(M,\omega)$ be a closed symplectically aspherical manifold. 
Filtered Floer homology of a Hamiltonian function on $(M,\omega)$ 
carries more information than its underlying Floer persistence barcode (see~\cite{PolterovichShelukhin,UsherZhang}).
Its action filtration records the birth and death of classes, while the
ordinary cohomology ring of the ambient manifold acts by Floer cap operators.
Persistence modules with operators provide a general language for exploiting
such additional structure in Morse and Floer theory
\cite{PolterovichShelukhinStojisavljevic}.  Persistent cup-length and
persistent cup modules provide a parallel multiplicative refinement for
filtered cohomology rings
\cite{ContessotoMemoliStefanouZhou,MemoliStefanouZhou}.  The purpose of this
paper is to organize the full ordinary cohomology action on Hamiltonian Floer
persistence and to extract a geometric consequence that is invisible to the
ordinary barcode and to all Schwarz selectors defined in ~\cite{Schwarz} which are also called \textit{Oh-Schwarz spectral invariants} in the literature.

Throughout, fixing a ground field $\K$,  we set
\[
 A:=H^*(M;\K),\qquad \mideal:=H^{>0}(M;\K).
\]
The cup product $u\smilecup v$ belongs to the ambient algebra $A$, whereas
\[
 HF_k(H)\otimes A^r\longrightarrow HF_{k-r}(H),
 \qquad (x,u)\longmapsto x\frownH u,
\]
is a module action on the $k$-th Floer homology group  $HF_k(H)$ of time-dependent Hamiltonian $H:[0,1]\times M\to\R$; see~\cite[Section~2.3]{Schwarz}.  Its associativity is
\[
 (x\frownH u)\frownH v=x\frownH(u\smilecup v).
\]
We therefore use cup terminology only for multiplication in $A$ and cap
terminology for the action on Floer homology.

Recall that an ideal $I$ in a graded ring is \textit{homogeneous} if, for every element $r\in I$, all of its homogeneous components also belong to $I$.  
For a homogeneous ideal $J\subseteq\mideal$, the depth-$\ell$ cap layer of a
graded $A$-module $V$ with action $A\times V\to V,\;(a,x)\mapsto x\frown a$ is the concrete subspace
\[
 \CapLayer_J^\ell(V)=\operatorname{span}_{\K}
 \{x\frown u_1\frown\cdots\frown u_\ell
   \mid x\in V,\ u_i\in J\}.
\]
Thus $\CapLayer_J^\ell(V)$ is generated by the outputs of all length-$\ell$
words whose letters lie in $J$.  
The total cap flag corresponds to $J=\mideal$, while a
principal ideal at depth one recovers the image persistence module of one cap
operator.

The main new geometric point is that ideal labels remember how a symplectic
mapping class acts on cohomology.  Hamiltonian conjugation acts trivially on
$H^*(M;\K)$ and hence preserves every decorated barcode.  A general
$\eta\in\Symp(M,\omega)$ instead sends the $J$-layer to the
$(\eta^{-1})^*J$-layer.  Consequently the cap profile gives a lower bound for
the Hofer distance between Hamiltonian-conjugacy orbits of $\phi$ and
$\eta\phi\eta^{-1}$.  This connects operator-enriched persistence with the
Hofer geometry of conjugacy classes, a theme already central in symplectic
geometry \cite{EntovConjugacy}.

The second geometric input is the genus-two construction of
Polterovich--Shelukhin--Stojisavljevi\'c
\cite[Section~2.4]{PolterovichShelukhinStojisavljevic}.  Their two Morse
functions are related by a handle-exchange diffeomorphism, have the same
ordinary persistence barcode and the same minimax value for every homology
class, but are separated by the image barcode of intersection with one fixed
class.  After choosing an invariant area form and passing to sufficiently
small autonomous Hamiltonians, this produces two symplectically conjugate
Hamiltonian diffeomorphisms with a positive cap-profile distance between
their Hamiltonian-conjugacy orbits.  Hence the new lower bound is not a
function of the ordinary Floer barcode together with the full collection of
Schwarz selectors.

The analytical ingredients are standard in the aspherical setting: Schwarz's
filtered cap-action, its strict action drop, and continuation estimates
\cite{Schwarz}; the small-Hamiltonian Morse--Floer
identification; and push-forward naturality.  We record the relevant energy
and compactification arguments so that every later filtration statement has a
precise scope.  

\subsection*{Main results}
Let
\[
 V_a^H:=HF_*^{(-\infty,a]}(H)\]
 be the filtered Floer homology of a time-dependent smooth Hamiltonian $H$ with the inclusion-induced homomorphism
 $i_a^b:V_a^H\longrightarrow V_b^H$; see Section~\ref{subsec:convention}.

For a graded $A$-module $V$, we define its \textit{persistent cap-length} as
\[
 \mcl_A(V):=\max\{\ell\ge1\mid \CapLayer^\ell(V)\ne0\},
\]
with \textbf{value zero} when the first total layer vanishes, and set
\[
 \PCL_H([a,b]):=\mcl_A(\im i_a^b)\in \mathbb{N}. 
\]
Here $ \mathbb{N}$ denotes the set of nonnegative integers. 

\begin{theorem}[Ideal-decorated cap persistence; see~Theorems~\ref{thm:PCL-basic} \&\ref{thm:universal-stability}]
\label{thm:intro-total}
The scalar invariant $\PCL_H$ is interval anti-monotone and Hofer stable.  For
every homogeneous ideal $J\subseteq A^{>0}$, depth $\ell\ge1$, and Floer
degree $q$, the assignment
\[
 a\longmapsto\bigl(\CapLayer_J^\ell(V_a^H)\bigr)_q
\]
is a constructible persistence module with barcode
$\Bcap_{J,\ell,q}(H)$.  For normalized Hamiltonians $H,K$,
\[
 d_{\mathrm B}\bigl(\Bcap_{J,\ell,q}(H),
                     \Bcap_{J,\ell,q}(K)\bigr)
 \leq\lVert H-K\rVert.
\]
Here
\(\lVert\cdot\rVert\)
denotes the Hofer norm of a Hamiltonian, defined in
\eqref{eq:Hofer-norm-Hamiltonian}.
The supremum over all $J,\ell,q$ defines $\Dcap$, and on
$\Ham(M,\omega)$ one has
\[
 \Dcap(\phi,\psi)\leq\distH(\phi,\psi).
\]
\end{theorem}

For $\phi\in\Ham(M,\omega)$, let
\[
 \OrbHam(\phi):=\{\theta\phi\theta^{-1}\mid
                    \theta\in\Ham(M,\omega)\}
\]
be its \emph{Hamiltonian-conjugacy orbit}, and write
$[\phi]_{\mathrm{Ham}}$ for this orbit.  We reserve
\[
 \Csymp{\eta}{\phi}:=\eta\phi\eta^{-1},
 \qquad \eta\in\Symp(M,\omega),
\]
for \emph{symplectic conjugation}.  The pseudodistance $\dHamConj$ is taken
between Hamiltonian-conjugacy orbits (cf.~\cref{eq:dconj}); the external conjugator $\eta$ need not
be Hamiltonian.

\begin{theorem}[see also Theorem~\ref{thm:conj-quotient}]
\label{thm:intro-conjugacy}
For $\eta\in\Symp(M,\omega)$, define
\[
 \DeltaCap_\eta(\phi):=
 \sup_{J,\ell,q}
 d_{\mathrm B}\bigl(\Bcap_{J,\ell,q}(\phi),
                     \Bcap_{\eta^*J,\ell,q}(\phi)\bigr).
\]
Then
\[
 \dHamConj\bigl([\phi]_{\mathrm{Ham}},[\Csymp{\eta}{\phi}]_{\mathrm{Ham}}\bigr)
 \geq \DeltaCap_\eta(\phi).
\]
The right-hand side depends only on the Hamiltonian-conjugacy orbit of
$\phi$ and on the induced automorphism of $\eta$ on $H^*(M;\K)$; in particular,
it is constant on symplectic isotopy classes of $\eta$.
\end{theorem}

If $\eta\in\Ham(M,\omega)$, both sides of the preceding inequality vanish;
all nontrivial applications therefore use a symplectic diffeomorphism outside
$\Ham(M,\omega)$.

\begin{theorem}[see also Theorem~\ref{thm:real-Hamiltonian}]
\label{thm:intro-real}
There exist a symplectic genus-two surface $(\Sigma_2,\sigma)$, an area-preserving handle-exchange diffeomorphism
$\eta\in\Symp(\Sigma_2,\sigma)\setminus\Ham(\Sigma_2,\sigma)$, and normalized
non-degenerate autonomous Hamiltonians $H_f,H_g$ with
\[
 \phi_{H_g}=\Csymp{\eta^{-1}}{\phi_{H_f}}
\]
such that their ordinary Floer barcodes agree and the corresponding Schwarz selectors coincide:
\[
 c(\alpha;H_f)=c(\alpha;H_g)
 \qquad\hbox{for all}\;\alpha\in H^*(\Sigma_2;\K)\setminus\{0\},
\]
but
\[
 \dHamConj\bigl([\phi_{H_f}]_{\mathrm{Ham}},[\phi_{H_g}]_{\mathrm{Ham}}\bigr)>0.
\]
  
\end{theorem}

\begin{remark}
From the theorem as above one can see that a principal depth-one cap barcode gives an explicit positive
lower bound.
Note that ordinary Floer persistence together with all Schwarz's
selectors does not determine the Hofer geometry of Hamiltonian conjugacy
orbits in this example.  The handle-exchange mapping class constructed in~\cref{thm:intro-real}  acts nontrivially
on the Hofer-pseudometric conjugacy quotient of $\Ham(\Sigma_2,\sigma)$.
\end{remark}

For completeness, in an atoroidal primitive free homotopy class we also
restrict loop rotation to each ideal-decorated cap layer.  The resulting
\textit{multiplicity-sensitive spreads} are $p$-Hofer-Lipschitz and vanish on full
$p$-th powers; see~\cref{sec:cyclic}.  The following structural result shows that the passage from
single operators to higher ideal layers is not merely notational.

\begin{theorem}[Formal strictness of higher cyclic cap layers; see
\cref{thm:formal-cyclic-strictness}]
There exist a closed symplectically aspherical manifold $M_0$, a graded
subalgebra $A_0\subset H^*(M_0;\mathbb Q)$, a homogeneous ideal
$J\subset A_0^{>0}$, and a constructible $\mathbb Z_2$-persistence module
$(\mathcal Q,R)$ in graded right $A_0$-modules for an automorphism $R:\mathcal Q\to \mathcal Q$
satisfying $R^2=\id$ such that the ordinary cyclic
spread and the spread of the image of every single homogeneous operator
vanish in every degree, whereas the depth-two $J$-layer has positive spread.
One may take $M_0=\Sigma_2\times\Sigma_3$ where $\Sigma_g$ denotes a closed oriented surface of genus $g$. 
\end{theorem}

The theorem proves strictness only in the algebraic equivariant-persistence
category.  It does not assert a new aspherical egg-beater example or a
strictly stronger Hofer-geometric obstruction; realizing this pattern by a
concrete Hamiltonian Floer module remains open.

The paper remains entirely in the closed symplectically aspherical setting.
All module operators come from ordinary cohomology, and no Novikov valuation
or quantum correction appears. The results of this paper will be further developed for monotone symplectic manifolds
in the forthcoming paper~\cite{Go}. 

\section{Cap layers in graded modules}
\label{sec:algebra}

Let $(A=\bigoplus_{r\ge0}A^r,+,\smilecup)$ be a finite-dimensional graded
commutative algebra over $\K$ with $A^0=\K\cdot1$, and let
$\mideal=A^{>0}=\bigoplus_{r>0}A^r$.  All modules are graded right
$A$-modules and all morphisms are degree-preserving and $A$-linear.

\subsection{Explicit cap-layer notation}

\begin{definition}[Cap layers]
\label{def:cap-layer}
Let $J\subseteq\mideal$ be a homogeneous ideal and let $V$ be a graded right
$A$-module with the action given by
\[
A\times V\longrightarrow V,\quad (a,x)\longmapsto x\frown a. 
\]

Define
\[
 \CapLayer_J^0(V):=V
\]
and, for $\ell\ge1$,
\begin{equation}
\label{eq:explicit-cap-layer}
 \CapLayer_J^\ell(V):=
 \operatorname{span}_{\K}
 \left\{((x\frown u_1)\cdots)\frown u_\ell
 \ \middle|\ x\in V,\ u_1,\ldots,u_\ell\in J\right\}.
\end{equation}
For a fixed $x\in V$, the notation $\CapLayer_J^\ell(x)$ means the same span with
the initial element fixed to $x$.  The \emph{total cap layer} is
\[
 \CapLayer^\ell(V):=\CapLayer_{\mideal}^\ell(V).
\]
\end{definition}

By the module law, $\CapLayer_J^\ell(V)$ is the image of the action map
$V\otimes J^\ell\to V$.  Thus one may write $VJ^\ell$ as shorthand, but
$ \CapLayer_J^\ell(V)$ is the definition used in this paper.  Since
$J^{\ell+1}\subseteq J^\ell$, the layers form a descending flag
\[
 V=\CapLayer_J^0(V)\supseteq\CapLayer_J^1(V)\supseteq
 \CapLayer_J^2(V)\supseteq\cdots.
\]
Because $J\subseteq A^{>0}$ and $A$ is finite-dimensional, this flag has
finite depth.

\begin{lemma}[A principal depth-one layer is an operator image]
\label{lem:principal}
Let $u\in A^{>0}$ be homogeneous and let $J_u=(u)=Au$ be the homogeneous
principal ideal generated by $u$.  Then
\[
 \CapLayer_{J_u}^1(V)=\im(m_u:V\to V),\qquad m_u(x)=x\frown u.
\]
\end{lemma}

\begin{proof}
Since $u\in J_u$, the right-hand side is contained in the left-hand side.
Conversely, every element of $J_u$ is a sum of terms $a\smilecup u$.  Hence
\[
 x\frown(a\smilecup u)=(x\frown a)\frown u\in\im m_u.
\]
Taking spans proves the reverse inclusion.
\end{proof}

\begin{definition}
For a graded right $A$-module $V$, we define \textit{module cap-length} of $V$ as 
\begin{equation}
\label{eq:mcl}
 \mcl_A(V):=\max\{\ell\in\N_{>0}\mid\CapLayer^\ell(V)\ne0\},
\end{equation}
with $\mcl_A(V)=0$ if $\CapLayer^1(V)=0$.  For an $A$-linear morphism
$f:V\to W$, set $\mcl_A(f):=\mcl_A(\im f)$.
\end{definition}

\begin{definition}
\label{def:induced-cap-layer-map}
Let \(f\colon V\to W\) be a degree-preserving homomorphism of
graded right \(A\)-modules. For every homogeneous ideal
\(J\subseteq A^{>0}\) and every integer \(\ell\geq 0\), define
\[
\operatorname{Cap}_{J}^{\ell}(f)
:=
\left.f\right|_{\operatorname{Cap}_{J}^{\ell}(V)}
\colon
\operatorname{Cap}_{J}^{\ell}(V)
\longrightarrow
\operatorname{Cap}_{J}^{\ell}(W).
\]
This restriction is well defined because \(f\) is \(A\)-linear.
Indeed, for every \(x\in V\) and \(u_1,\ldots,u_\ell\in J\),
\[
\begin{aligned}
f\bigl(
 ((x\frown u_1)\cdots)\frown u_\ell
\bigr)
=
((f(x)\frown u_1)\cdots)\frown u_\ell
\in
\operatorname{Cap}_{J}^{\ell}(W).
\end{aligned}
\]
For \(\ell=0\), this convention gives
\[
\operatorname{Cap}_{J}^{0}(f)=f.
\]
\end{definition}

\begin{lemma}[Epi--mono behavior]
\label{lem:epi-mono}
Let $f:V\to W$ be $A$-linear.
\begin{enumerate}
\item If $f$ is injective, then $\mcl_A(V)\le\mcl_A(W)$.
\item If $f$ is surjective, then $\mcl_A(V)\ge\mcl_A(W)$.
\end{enumerate}
\end{lemma}

\begin{proof}
For every $\ell$, $A$-linearity gives
\[
 f\bigl(\CapLayer^\ell(V)\bigr)=\CapLayer^\ell(\im f).
\]
If $f$ is injective, a nonzero vector in $\CapLayer^\ell(V)$ has nonzero image in
$\CapLayer^\ell(W)$.  If $f$ is surjective, the displayed equality becomes
$f(\CapLayer^\ell(V))=\CapLayer^\ell(W)$, so nonvanishing on the right implies
nonvanishing on the left.
\end{proof}

\begin{proposition}[Categorical cap-length]
\label{prop:categorical}
For composable $A$-linear maps $V\xrightarrow{f}W\xrightarrow{g}Z$,
\[
 \mcl_A(gf)\le\min\{\mcl_A(f),\mcl_A(g)\}.
\]
\end{proposition}

\begin{proof}
There is an $A$-linear surjection $\im f\twoheadrightarrow\im(gf)$ and an
$A$-linear injection $\im(gf)\hookrightarrow\im g$.  Then
\cref{lem:epi-mono} concludes the desired inequality.
\end{proof}

Let
\[
 \cl(A):=\max\{\ell\ge1\mid
 u_1\smilecup\cdots\smilecup u_\ell\ne0
 \text{ for some }u_i\in A^{>0}\}
\]
denote the \textit{cup-length} of $A$ with respect to the product $\smilecup$. 
Then $\mcl_A(V)\le\cl(A)$ for every $V$. If we regard \(A\) as its \textit{right regular module} \(A_A\), whose
right action is given by multiplication:
\[
a\cdot u:=a\smile u,
\qquad a,u\in A,
\]
then
\[
\operatorname{Cap}^{\ell}(A_A)
=
A\mathfrak m^\ell
=
\mathfrak m^\ell,
\]
and consequently
\[
\mcl_A(A_A)
=
\max\{\ell\geq0\mid\mathfrak m^\ell\neq0\}
=
\cl(A).
\]

\subsection{Persistence modules of graded \(A\)-modules}

Let
\[
A=\bigoplus_{r\geq 0} A^r
\]
be a finite-dimensional, nonnegatively graded,
graded-commutative algebra over a field \(\mathbb{K}\), satisfying
$
A^0=\mathbb{K}\cdot 1.
$

A graded right \(A\)-module is a
\(\mathbb{Z}\)-graded \(\mathbb{K}\)-vector space
\[
V=\bigoplus_{q\in\mathbb{Z}} V_q
\]
equipped with a right \(A\)-action
\[
V_q\times A^r\longrightarrow V_{q-r},
\qquad
(x,u)\longmapsto x\frown u,
\]
for all \(q\in\mathbb{Z}\) and \(r\geq 0\).
Thus, multiplication by an element of cohomological degree \(r\)
lowers the module degree by \(r\). This is the grading convention
used for the Floer cap action.

We denote by
\[
\mathrm{GrMod}_{A}
\]
the category whose objects are graded right \(A\)-modules with the
above grading convention. A morphism
\[
F\colon V\longrightarrow W
\]
in \(\mathrm{GrMod}_{A}\) is a degree-preserving right
\(A\)-linear map. Explicitly, it satisfies
\[
F(V_q)\subseteq W_q
\qquad
\text{for every }q\in\mathbb{Z},
\]
and
\[
F(x\cdot u)=F(x)\cdot u
\]
for all \(x\in V\) and \(u\in A\).

We regard the partially ordered set
\(
(\mathbb{R},\leq)
\)
as a category whose objects are real numbers and in which there is
a unique morphism \(a\to b\) precisely when \(a\leq b\).

\begin{definition}\label{def:persistence-module}
An \textit{\(\mathbb{R}\)-indexed persistence module} of graded right
\(A\)-modules is a functor
\[
\mathcal{V}\colon
(\mathbb{R},\leq)\longrightarrow \mathrm{GrMod}_{A}.
\]
Equivalently, it consists of the following data:
\begin{enumerate}
\item for every \(a\in\mathbb{R}\), a graded right \(A\)-module
\[
V_a:=\mathcal{V}(a);
\]

\item for every \(a\leq b\), a degree-preserving right
\(A\)-linear structure map
\[
v_a^b
:=
\mathcal{V}(a\leq b)
\colon V_a\longrightarrow V_b;
\]

\item the identities
\[
v_a^a=\operatorname{id}_{V_a}
\]
and
\[
v_b^c\circ v_a^b=v_a^c
\qquad
\text{for all }a\leq b\leq c.
\]
\end{enumerate}

We write
\[
\mathcal{V}
=
\bigl(V_a,v_a^b\bigr)_{a\leq b}
\]
when the objects and structure maps are to be displayed explicitly.
\end{definition}

\begin{remark}
The functorial conditions in
\cref{def:persistence-module} define an arbitrary persistence module.
Following Patel \cite[Definition~2.2]{PatelGPD}, we call such a
module \textit{constructible} when it is constant up to isomorphism away
from a finite set of critical values.
\end{remark}

\begin{definition}[\(S\)-constructible persistence module]
\label{def:constructible-persistence-module}
Let
\[
S=\{s_1<\cdots<s_N\}\subset\mathbb{R}
\]
be finite.  A persistence module
\[
\mathcal V=(V_a,v_a^b)_{a\leq b}
\colon
(\mathbb R,\leq)\longrightarrow\mathrm{GrMod}_A
\]
is called \textit{\(S\)-constructible} if:

\begin{enumerate}
\item each graded component \((V_a)_q\) is finite-dimensional over
\(\mathbb K\);

\item for every \(a\leq b\) belonging to the same connected
component of \(\mathbb R\setminus S\), the structure map
\[
v_a^b\colon V_a\longrightarrow V_b
\]
is an isomorphism of graded right \(A\)-modules.
\end{enumerate}

The elements of \(S\) are called \textit{critical values}.  We say that
\(\mathcal V\) is \textit{constructible} if it is \(S\)-constructible for
some finite set \(S\subset\mathbb R\).
\end{definition}

\medskip
\noindent\textbf{Interval modules and barcodes.}
We next recall the barcode of an ordinary one-parameter persistence
module over $\mathbb K$; see
\cite{CrawleyBoevey,ChazalEtAl}.  Let
\(
 \mathrm{Vect}_{\mathbb K}
\)
denote the category of $\mathbb K$-vector spaces, and let
$\mathcal U=(U_a,u_a^b)_{a\leq b}\colon
(\mathbb R,\leq)\to\mathrm{Vect}_{\mathbb K}$ be a persistence
module.

For a nonempty interval $I\subseteq\mathbb R$, the \emph{interval
module} $\mathbb K_I$ is defined by
\[
 (\mathbb K_I)_a=
 \begin{cases}
  \mathbb K,& a\in I,\\
  0,& a\notin I,
 \end{cases}
\]
and, for $a\leq b$, by the structure map
\[
 \iota_a^b=
 \begin{cases}
  \operatorname{id}_{\mathbb K},& a,b\in I,\\
  0,&\text{otherwise}.
 \end{cases}
\]
If $\mathcal U$ is pointwise finite-dimensional, then the interval
decomposition theorem of Crawley--Boevey
\cite{CrawleyBoevey} gives an isomorphism
\begin{equation}
\label{eq:interval-decomposition}
 \mathcal U\cong\bigoplus_{I\in\operatorname{barc}(\mathcal U)}
 \mathbb K_I,
\end{equation}
where the collection of intervals is a multiset and is uniquely
determined up to permutation.  This multiset is called the
\emph{barcode} of $\mathcal U$ and is denoted by
\[
 \operatorname{barc}(\mathcal U).
\]
In particular, every pointwise finite-dimensional constructible
persistence module admits such a barcode.  Repeated intervals in a
barcode are always counted with multiplicity.

\medskip
\noindent\textbf{The bottleneck distance.}
For a nonempty interval $I\subseteq\mathbb R$, write
\[
 b(I):=\inf I\in\mathbb R\cup\{-\infty\},
 \qquad
 d(I):=\sup I\in\mathbb R\cup\{+\infty\}.
\]
The symbols $b(I)$ and $d(I)$ record the numerical endpoints; the
open or closed endpoint decorations do not affect the numerical errors
below.  Thus intervals that differ only by endpoint decoration may
have distance zero, as is customary for the interleaving/bottleneck
pseudometric on real-indexed persistence modules
\cite{ChazalEtAl}.

For intervals $I$ and $J$, define their endpoint cost by
\[
 \delta_\infty(I,J):=
 \begin{cases}
  \max\{
    |b(I)-b(J)|,\;
    |d(I)-d(J)|
\},
   & I,J\text{ bounded},\\[1mm]
  |b(I)-b(J)|,
   & I,J\text{ unbounded above only},\\[1mm]
  |d(I)-d(J)|,
   & I,J\text{ unbounded below only},\\[1mm]
  0,
   & I=J=\mathbb R,\\[1mm]
  +\infty,&\text{otherwise}.
 \end{cases}
\]
The cost of leaving $I$ unmatched, equivalently of matching it to
the diagonal (a nonempty closed interval whose two endpoints of  are the same), is
\[
 \delta_\Delta(I):=
 \begin{cases}
  \dfrac{d(I)-b(I)}{2},& I\text{ bounded},\\[2mm]
  +\infty,&\text{otherwise}.
 \end{cases}
\]
Thus an interval unbounded in either direction cannot be matched to
the diagonal.

Let $\mathcal B$ and $\mathcal C$ be barcodes.  A \emph{partial
matching} $\mu$ between them is a bijection
\[
 \mu\colon\mathcal B'\longrightarrow\mathcal C'
\]
between submultisets $\mathcal B'\subseteq\mathcal B$ and
$\mathcal C'\subseteq\mathcal C$.  Its error is
\[
 \operatorname{Err}(\mu):=
 \max\left\{
 \begin{array}{l}
  \displaystyle
  \sup_{I\in\mathcal B'}\delta_\infty(I,\mu(I)),\\[2mm]
  \displaystyle
  \sup_{I\in\mathcal B\setminus\mathcal B'}\delta_\Delta(I),\\[2mm]
  \displaystyle
  \sup_{J\in\mathcal C\setminus\mathcal C'}\delta_\Delta(J)
 \end{array}
 \right\},
\]
where the supremum over an empty multiset is understood to be zero.
The \emph{bottleneck distance} between $\mathcal B$ and
$\mathcal C$ is
\begin{equation}
\label{eq:bottleneck-distance}
 d_{\mathrm B}(\mathcal B,\mathcal C)
 :=\inf_{\mu}\operatorname{Err}(\mu),
\end{equation}
where the infimum runs over all partial matchings.  In particular, a
bounded bar $I$ may be left unmatched at cost
$(d(I)-b(I))/2$, whereas a bar unbounded above must be matched to
another bar unbounded above.

\begin{definition}
For $a\le b$, define
\[
 \PCL_{\cV}([a,b]):=\mcl_A(\im v_a^b).
\]
For a homogeneous ideal $J\subseteq\mideal$, depth $\ell\ge0$, and degree
$q$, define the fixed layer
\[
 \Psi_{J}^{\ell,q}(\cV)_a:=\bigl(\CapLayer_J^\ell(V_a)\bigr)_q.
\]
\end{definition}

\begin{lemma}[Functoriality of cap layers]
\label{lem:layer-functor}
Let \(f\colon V\to W\) be a degree-preserving \(A\)-linear map.
For every homogeneous ideal \(J\subseteq A^{>0}\) and every
\(\ell\geq0\), the induced map
\[
\CapLayer_J^\ell(f)
\colon
\CapLayer_J^\ell(V)
\longrightarrow
\CapLayer_J^\ell(W)
\]
satisfies
\begin{equation}
\label{eq:cap-functor}
\operatorname{im}\bigl(\CapLayer_J^\ell(f)\bigr)
=
f\bigl(\CapLayer_J^\ell(V)\bigr)
=
\CapLayer_J^\ell(\operatorname{im}f).
\end{equation}
In particular,
\[
f\bigl(\CapLayer_J^\ell(V)\bigr)
\subseteq
\CapLayer_J^\ell(W).
\]
\end{lemma}

\begin{proof}
By the definition of the restricted map,
\[
\operatorname{im}\bigl(\CapLayer_J^\ell(f)\bigr)
=
f\bigl(\CapLayer_J^\ell(V)\bigr).
\]
It remains to prove
\[
f\bigl(\CapLayer_J^\ell(V)\bigr)
=
\CapLayer_J^\ell(\operatorname{im}f).
\]

Let
\[
y=
((x\frown u_1)\cdots)\frown u_\ell
\]
be a generator of \(\CapLayer_J^\ell(V)\), where
\(x\in V\) and \(u_1,\ldots,u_\ell\in J\). Since \(f\) is
\(A\)-linear,
\[
f(y)
=
((f(x)\frown u_1)\cdots)\frown u_\ell,
\]
which belongs to \(\CapLayer_J^\ell(\operatorname{im}f)\).
Hence
\[
f\bigl(\CapLayer_J^\ell(V)\bigr)
\subseteq
\CapLayer_J^\ell(\operatorname{im}f).
\]

Conversely, every generator of
\(\CapLayer_J^\ell(\operatorname{im}f)\) has the form
\[
((z\frown u_1)\cdots)\frown u_\ell
\]
with \(z\in\operatorname{im}f\). Choose \(x\in V\) such that
\(z=f(x)\). Then
\[
\begin{aligned}
((z\frown u_1)\cdots)\frown u_\ell
&=
((f(x)\frown u_1)\cdots)\frown u_\ell\\
&=
f\bigl(
 ((x\frown u_1)\cdots)\frown u_\ell
\bigr),
\end{aligned}
\]
so this generator lies in
\(f(\CapLayer_J^\ell(V))\). This proves the reverse inclusion.
\end{proof}

\begin{proposition}[Interval anti-monotonicity]
\label{prop:anti-recover}
If $c\le a\le b\le d$, then
\[
 \PCL_{\cV}([a,b])\ge\PCL_{\cV}([c,d]).
\]
Moreover,
\begin{equation}
\label{eq:recover}
 \PCL_{\cV}([a,b])=
 \max\{\ell\ge1\mid
 \im(\CapLayer^\ell(V_a)\to\CapLayer^\ell(V_b))\ne0\}.
\end{equation}
\end{proposition}

\begin{proof}
The first assertion follows from \cref{prop:categorical} applied to the
factorization of $v_c^d$ through $v_a^b$.  By
\cref{lem:layer-functor},
\[
 \im(\CapLayer^\ell(V_a)\to\CapLayer^\ell(V_b))
 =\CapLayer^\ell(\im v_a^b),
\]
which is nonzero exactly when $\mcl_A(\im v_a^b)\ge\ell$.
\end{proof}

\subsection{Interleavings, erosion distance, and barcodes}
\label{subsec:interleavings-erosion}

We first recall the definitions of erosion distance and
interleaving distance.  Throughout this subsection, let
\[
\operatorname{Int}
:=
\bigl\{[a,b]\subset\mathbb{R}\mid a\leq b\bigr\}
\]
denote the set of nonempty closed bounded intervals in
\(\mathbb{R}\), ordered by inclusion.

Let \((P,\geq_P)\) be a partially ordered set.  A function
\[
p\colon \operatorname{Int}\longrightarrow P
\]
is called \textit{anti-monotone} if
\[
I\subseteq J
\quad\Longrightarrow\quad
p(I)\geq_P p(J).
\]
For \(\varepsilon\geq 0\) and \(I=[a,b]\in\operatorname{Int}\),
define the \textit{\(\varepsilon\)-enlargement} of \(I\) by
\[
I^\varepsilon
:=
[a-\varepsilon,b+\varepsilon].
\]

We use the standard notions of interleaving and erosion distance;
see, for example,
\cite[Definitions~3.1 and~3.2]{MemoliStefanouZhou}.
\begin{definition}[Erosion distance]
\label{def:erosion-distance}
Let
\[
p,q\colon\operatorname{Int}\longrightarrow P
\]
be anti-monotone functions.  We say that \(p\) and \(q\) are
\textit{\(\varepsilon\)-eroded} if, for every
\(I\in\operatorname{Int}\),
\[
p(I)\geq_P q(I^\varepsilon)
\qquad\text{and}\qquad
q(I)\geq_P p(I^\varepsilon).
\]
Their \textit{erosion distance} is
\[
d_{\mathrm E}(p,q)
:=
\inf
\left\{
\varepsilon\geq 0
\ \middle|\
p\text{ and }q\text{ are }\varepsilon\text{-eroded}
\right\}.
\]
If no such \(\varepsilon\) exists, we set
\[
d_{\mathrm E}(p,q):=+\infty.
\]
\end{definition}

Let
\[
\mathcal V=(V_a,v_a^b)_{a\leq b},
\qquad
\mathcal W=(W_a,w_a^b)_{a\leq b}
\]
be persistence modules in \(\mathrm{GrMod}_A\).

For \(\varepsilon\geq0\), the
\(\varepsilon\)-translation of \(\mathcal V\) is the persistence
module \(T_\varepsilon\mathcal V\) defined by
\[
(T_\varepsilon\mathcal V)_a
:=
V_{a+\varepsilon}
\]
and
\[
(T_\varepsilon v)_a^b
:=
v_{a+\varepsilon}^{b+\varepsilon},
\qquad a\leq b.
\]
There is a canonical natural transformation
\[
\eta_{\mathcal V}^{\varepsilon}
\colon
\mathcal V\Longrightarrow T_\varepsilon\mathcal V
\]
whose component at \(a\) is the structure map
\[
(\eta_{\mathcal V}^{\varepsilon})_a
=
v_a^{a+\varepsilon}.
\]

\begin{definition}[\(A\)-linear interleaving]
\label{def:A-linear-interleaving}
An \textit{\(A\)-linear \(\varepsilon\)-interleaving} between
\(\mathcal V\) and \(\mathcal W\) consists of two natural
transformations
\[
F\colon
\mathcal V\Longrightarrow T_\varepsilon\mathcal W,
\qquad
G\colon
\mathcal W\Longrightarrow T_\varepsilon\mathcal V,
\]
such that every component
\[
F_a\colon V_a\longrightarrow W_{a+\varepsilon},
\qquad
G_a\colon W_a\longrightarrow V_{a+\varepsilon}
\]
is a degree-preserving right \(A\)-linear map.

Naturality means that, for every \(a\leq b\),
\[
w_{a+\varepsilon}^{b+\varepsilon}\circ F_a
=
F_b\circ v_a^b
\]
and
\[
v_{a+\varepsilon}^{b+\varepsilon}\circ G_a
=
G_b\circ w_a^b.
\]
The interleaving identities are
\[
G_{a+\varepsilon}\circ F_a
=
v_a^{a+2\varepsilon}
\]
and
\[
F_{a+\varepsilon}\circ G_a
=
w_a^{a+2\varepsilon}
\]
for every \(a\in\mathbb{R}\).

Equivalently, in terms of natural transformations,
\[
(T_\varepsilon G)\circ F
=
\eta_{\mathcal V}^{2\varepsilon},
\qquad
(T_\varepsilon F)\circ G
=
\eta_{\mathcal W}^{2\varepsilon}.
\]
\end{definition}

The \textit{\(A\)-linear interleaving distance} is
\[
d_{\mathrm I}^{A}(\mathcal V,\mathcal W)
:=
\inf
\left\{
\varepsilon\geq0
\ \middle|\
\mathcal V\text{ and }\mathcal W
\text{ admit an \(A\)-linear }
\varepsilon\text{-interleaving}
\right\},
\]
with the convention that
\(
d_{\mathrm I}^{A}(\mathcal V,\mathcal W)=+\infty
\)
if no \(A\)-linear interleaving exists.

The persistent cap-length functions
\[
\operatorname{PCL}_{\mathcal V},
\operatorname{PCL}_{\mathcal W}
\colon
\operatorname{Int}\longrightarrow\mathbb{N}
\]
are anti-monotone by
\cref{prop:anti-recover}.  Hence their erosion distance is
well defined in the sense of
\cref{def:erosion-distance}.

\begin{theorem}[Abstract stability]
\label{thm:abstract-stability}
If \(\mathcal V\) and \(\mathcal W\) admit an \(A\)-linear
\(\varepsilon\)-interleaving, then
\[
d_{\mathrm E}
\bigl(
\operatorname{PCL}_{\mathcal V},
\operatorname{PCL}_{\mathcal W}
\bigr)
\leq\varepsilon.
\]
Consequently,
\[
d_{\mathrm E}
\bigl(
\operatorname{PCL}_{\mathcal V},
\operatorname{PCL}_{\mathcal W}
\bigr)
\leq
d_{\mathrm I}^{A}(\mathcal V,\mathcal W).
\]
Moreover, the interleaving restricts to every fixed-depth,
fixed-degree persistence module
\[
\Psi_J^{\ell,q}(\mathcal V)
\quad\text{and}\quad
\Psi_J^{\ell,q}(\mathcal W).
\]
\end{theorem}
\begin{proof}
Fix \(I=[a,b]\in\operatorname{Int}\).  By naturality and the
interleaving identities, the composite
\[
W_{a-\varepsilon}
\xrightarrow{\,G_{a-\varepsilon}\,}
V_a
\xrightarrow{\,v_a^b\,}
V_b
\xrightarrow{\,F_b\,}
W_{b+\varepsilon}
\]
is the structure map
\[
w_{a-\varepsilon}^{b+\varepsilon}.
\]
Indeed,
\[
\begin{aligned}
F_b\circ v_a^b\circ G_{a-\varepsilon}
&=
w_{a+\varepsilon}^{b+\varepsilon}
 \circ F_a\circ G_{a-\varepsilon}\\
&=
w_{a+\varepsilon}^{b+\varepsilon}
 \circ w_{a-\varepsilon}^{a+\varepsilon}\\
&=
w_{a-\varepsilon}^{b+\varepsilon}.
\end{aligned}
\]
Since
\[
I^\varepsilon
=
[a-\varepsilon,b+\varepsilon],
\]
the categorical monotonicity of module cap-length gives
\[
\operatorname{PCL}_{\mathcal W}(I^\varepsilon)
\leq
\operatorname{PCL}_{\mathcal V}(I).
\]
Exchanging \(\mathcal V\) and \(\mathcal W\) gives
\[
\operatorname{PCL}_{\mathcal V}(I^\varepsilon)
\leq
\operatorname{PCL}_{\mathcal W}(I).
\]
Thus the two persistent cap-length functions are
\(\varepsilon\)-eroded, and hence
\[
d_{\mathrm E}
\bigl(
\operatorname{PCL}_{\mathcal V},
\operatorname{PCL}_{\mathcal W}
\bigr)
\leq\varepsilon.
\]

For the last assertion, the \(A\)-linearity of \(F_a\) gives
\[
F_a
\bigl(
\operatorname{Cap}_J^\ell(V_a)
\bigr)
\subseteq
\operatorname{Cap}_J^\ell(W_{a+\varepsilon}),
\]
and similarly,
\[
G_a
\bigl(
\operatorname{Cap}_J^\ell(W_a)
\bigr)
\subseteq
\operatorname{Cap}_J^\ell(V_{a+\varepsilon}).
\]
Therefore \(F\) and \(G\) restrict to natural transformations
between the fixed-depth modules.  Since they preserve the grading,
they further restrict to the degree-\(q\) components.  The
interleaving identities remain valid after these restrictions.
\end{proof}

Assume the fixed layers are constructible and pointwise finite-dimensional.
Set
\[
 \Bcap_{J,\ell,q}(\cV):=
 \barc\bigl(\Psi_J^{\ell,q}(\cV)\bigr).
\]
For $J=\mideal$, abbreviate $\Bcap_{\ell,q}(\cV)$.

\begin{theorem}[Barcode stability and scalar recovery]
\label{thm:barcode}
For the total layers,
\[
 \PCL_{\cV}([a,b])=
 \max\left\{\ell\ge1\ \middle|\
 \begin{array}{c}
 \text{some bar in }\bigsqcup_q\Bcap_{\ell,q}(\cV)\\
 \text{contains }[a,b]
 \end{array}\right\}.
\]
If $\cV$ and $\cW$ are $A$-linearly $\eps$-interleaved, then
\[
 \sup_{J,\ell,q}d_{\mathrm B}
 \bigl(\Bcap_{J,\ell,q}(\cV),\Bcap_{J,\ell,q}(\cW)\bigr)
 \le\eps.
\]
\end{theorem}

\begin{proof}
We prove the scalar recovery statement and the barcode stability
statement separately.

\smallskip
\noindent
\emph{Step 1: recovery of the scalar persistent cap-length.}
Fix a closed bounded interval
\[
I=[a,b],\qquad a\leq b.
\]
By \eqref{eq:recover},
\begin{equation}
\label{eq:barcode-proof-recovery-1}
\PCL_{\cV}([a,b])
=
\max\left\{
\ell\geq1
\ \middle|\
\operatorname{im}
\left(
\CapLayer^\ell(V_a)
\longrightarrow
\CapLayer^\ell(V_b)
\right)
\neq0
\right\}.
\end{equation}
Here the map in \eqref{eq:barcode-proof-recovery-1} is
\[
\CapLayer^\ell(v_a^b)
=
\left.
v_a^b
\right|_{\CapLayer^\ell(V_a)}
\colon
\CapLayer^\ell(V_a)
\longrightarrow
\CapLayer^\ell(V_b).
\]

Since \(v_a^b\) preserves the grading, so does
\(\CapLayer^\ell(v_a^b)\).  Consequently,
\[
\operatorname{im}
\bigl(
\CapLayer^\ell(v_a^b)
\bigr)
\neq0
\]
if and only if there exists \(q\in\mathbb Z\) such that
\begin{equation}
\label{eq:barcode-proof-recovery-2}
\operatorname{im}
\left(
\bigl(\CapLayer^\ell(v_a^b)\bigr)_q
\right)
\neq0.
\end{equation}
By the definition of the fixed-depth, fixed-degree persistence
module, the map occurring in
\eqref{eq:barcode-proof-recovery-2} is precisely the structure map
of
\[
\Psi_{\mideal}^{\ell,q}(\cV)
\]
associated with the relation \(a\leq b\):
\[
\Psi_{\mideal}^{\ell,q}(\cV)(a\leq b)
=
\bigl(\CapLayer^\ell(v_a^b)\bigr)_q.
\]

Because \(\cV\) is constructible, each
\(\Psi_{\mideal}^{\ell,q}(\cV)\) is a pointwise
finite-dimensional constructible persistence module.  Hence it has
an interval decomposition
\begin{equation}
\label{eq:barcode-proof-decomposition}
\Psi_{\mideal}^{\ell,q}(\cV)
\cong
\bigoplus_{K\in\Bcap_{\ell,q}(\cV)}
\K_K,
\end{equation}
where intervals $K$ are counted with multiplicity.

Under the decomposition
\eqref{eq:barcode-proof-decomposition}, the structure map from
parameter \(a\) to parameter \(b\) is the direct sum of the
corresponding structure maps of the interval summands:
\[
\bigoplus_{K\in\Bcap_{\ell,q}(\cV)}
\left(
(\K_K)_a\longrightarrow(\K_K)_b
\right).
\]
For a single interval module \(\K_K\), this structure map is nonzero
if and only if
\[
a\in K
\qquad\text{and}\qquad
b\in K.
\]
Since \(K\) is an interval and \(a\leq b\), this is equivalent to
\[
[a,b]\subseteq K.
\]
It follows that
\begin{align}
&\operatorname{im}
\left(
\bigl(\CapLayer^\ell(v_a^b)\bigr)_q
\right)
\neq0
\nonumber\\
&\qquad\Longleftrightarrow
\text{there exists a bar }
K\in\Bcap_{\ell,q}(\cV)
\text{ such that }[a,b]\subseteq K.
\label{eq:barcode-proof-recovery-3}
\end{align}

Combining
\eqref{eq:barcode-proof-recovery-1},
\eqref{eq:barcode-proof-recovery-2}, and
\eqref{eq:barcode-proof-recovery-3}, we obtain
\[
\PCL_{\cV}([a,b])
=
\max\left\{
\ell\geq1
\ \middle|\
\begin{array}{c}
\text{there exist \(q\in\mathbb Z\) and a bar}\\
K\in\Bcap_{\ell,q}(\cV)
\text{ with }[a,b]\subseteq K
\end{array}
\right\}.
\]
Equivalently,
\[
\PCL_{\cV}([a,b])
=
\max\left\{
\ell\geq1
\ \middle|\
\begin{array}{c}
\text{some bar in }
\bigsqcup_q\Bcap_{\ell,q}(\cV)\\
\text{contains }[a,b]
\end{array}
\right\}.
\]
As usual, if no such \(\ell\geq1\) exists, the maximum on the
right-hand side is understood to be \(0\), consistently with the
definition of \(\PCL_{\cV}\).

\smallskip
\noindent
\emph{Step 2: stability of every decorated barcode.}
Assume now that \(\cV\) and \(\cW\) are \(A\)-linearly
\(\eps\)-interleaved.  Fix a homogeneous ideal
\[
J\subseteq\mideal,
\]
a depth \(\ell\geq0\), and a degree \(q\in\mathbb Z\).

Let
\[
F_a\colon V_a\longrightarrow W_{a+\eps},
\qquad
G_a\colon W_a\longrightarrow V_{a+\eps}
\]
be the components of the given \(A\)-linear
\(\eps\)-interleaving.  Since \(F_a\) and \(G_a\) are
degree-preserving and \(A\)-linear, they satisfy
\[
F_a\bigl(\CapLayer_J^\ell(V_a)\bigr)
\subseteq
\CapLayer_J^\ell(W_{a+\eps})
\]
and
\[
G_a\bigl(\CapLayer_J^\ell(W_a)\bigr)
\subseteq
\CapLayer_J^\ell(V_{a+\eps}).
\]
Therefore they restrict to degree-\(q\) maps
\[
F_a^{J,\ell,q}
\colon
\bigl(\CapLayer_J^\ell(V_a)\bigr)_q
\longrightarrow
\bigl(\CapLayer_J^\ell(W_{a+\eps})\bigr)_q
\]
and
\[
G_a^{J,\ell,q}
\colon
\bigl(\CapLayer_J^\ell(W_a)\bigr)_q
\longrightarrow
\bigl(\CapLayer_J^\ell(V_{a+\eps})\bigr)_q.
\]

Equivalently,
\[
F_a^{J,\ell,q}
\colon
\Psi_J^{\ell,q}(\cV)_a
\longrightarrow
\Psi_J^{\ell,q}(\cW)_{a+\eps}
\]
and
\[
G_a^{J,\ell,q}
\colon
\Psi_J^{\ell,q}(\cW)_a
\longrightarrow
\Psi_J^{\ell,q}(\cV)_{a+\eps}.
\]
The naturality identities for \(F\) and \(G\) remain valid after
restriction.  Moreover, restricting the original interleaving
identities gives
\[
G_{a+\eps}^{J,\ell,q}
\circ
F_a^{J,\ell,q}
=
\bigl(
\CapLayer_J^\ell(v_a^{a+2\eps})
\bigr)_q
\]
and
\[
F_{a+\eps}^{J,\ell,q}
\circ
G_a^{J,\ell,q}
=
\bigl(
\CapLayer_J^\ell(w_a^{a+2\eps})
\bigr)_q.
\]
The maps on the right-hand sides are exactly the
\(2\eps\)-structure maps of
\(\Psi_J^{\ell,q}(\cV)\) and
\(\Psi_J^{\ell,q}(\cW)\), respectively.  Thus
\[
\Psi_J^{\ell,q}(\cV)
\qquad\text{and}\qquad
\Psi_J^{\ell,q}(\cW)
\]
are ordinarily \(\eps\)-interleaved.

Since these modules are constructible and pointwise
finite-dimensional, the algebraic stability theorem, equivalently
the isometry theorem for one-parameter persistence modules, yields
\[
d_{\mathrm B}
\left(
\barc\bigl(\Psi_J^{\ell,q}(\cV)\bigr),
\barc\bigl(\Psi_J^{\ell,q}(\cW)\bigr)
\right)
\leq\eps;
\]
see \cite{ChazalEtAl,BauerLesnick}.  By the definition of the
cap-decorated barcodes, this is
\[
d_{\mathrm B}
\left(
\Bcap_{J,\ell,q}(\cV),
\Bcap_{J,\ell,q}(\cW)
\right)
\leq\eps.
\]
The estimate holds for every homogeneous ideal \(J\), every
\(\ell\geq0\), and every \(q\in\mathbb Z\).  Taking the supremum
over all triples \((J,\ell,q)\) therefore gives
\[
\sup_{J,\ell,q}
d_{\mathrm B}
\left(
\Bcap_{J,\ell,q}(\cV),
\Bcap_{J,\ell,q}(\cW)
\right)
\leq\eps.
\]
This proves both assertions.
\end{proof}

\begin{example}
\label{ex:strict-refinement}
Let $A=\K[u]/(u^4)$ and consider constant persistence modules 
\[
 V=A\oplus\K\oplus\K,
 \qquad W=A\oplus A/(u^2),
\]
where each copy of \(\mathbb K\) is regarded as the
\(A\)-module \(A/(u)\); thus \(1\in A\) acts as the identity and
\(u\) acts trivially.  The underlying vector
spaces both have dimension six and both modules have cap-length three.
However,
\[
 \dim_{\mathbb K}\CapLayer^1(V)=3,\qquad \dim_{\mathbb K}\CapLayer^1(W)=4.
\]
Thus the ordinary barcode and scalar cap-length agree, while the first total
cap layer differs. This is an example explain why the total flag is stronger than scalar cap-length. 
\end{example}
\section{Floer-theoretic input}
\label{sec:floer}

\subsection{Conventions}\label{subsec:convention}
Let $(M^{2n},\omega)$ be closed and connected and assume
\begin{equation}
\label{eq:aspherical}
 \omega|_{\pi_2(M)}=0,
 \qquad
 c_1(TM)|_{\pi_2(M)}=0.
\end{equation}
Fix a coefficient field $\K$ for which the relevant Floer complexes are
oriented; one may always take $\K=\mathbb F_2$.  Let $\Lambda M$ denote the set of all contractible loops $x:S^1:=\mathbb{R}/\mathbb{Z}\to M$. 

For a smooth time-dependent Hamiltonian
\[
F\colon S^1\times M\longrightarrow\mathbb R,
\]
write
\[
\operatorname{osc}_M(F_t)
:=
\max_{x\in M}F_t(x)-\min_{x\in M}F_t(x).
\]
Its \(L^{(1,\infty)}\)-oscillation norm is
\begin{equation}
\label{eq:Hofer-norm-Hamiltonian}
\lVert F\rVert
:=
\int_0^1\operatorname{osc}_M(F_t)\,dt
=
\int_0^1
\left(
\max_M F_t-\min_M F_t
\right)\,dt.
\end{equation}
On the space of all Hamiltonians this is a seminorm, since it is
unchanged by adding a function of time only.  On the space of
normalized Hamiltonians, i.e. $\int_M H_t \omega=0$ for all $t\in S^1$, it is a genuine norm; see \cite{Ho,Po}.  We refer to it as
the \textit{Hofer norm} of a Hamiltonian, which 
will be also reserved for the corresponding norm on
the Hamiltonian diffeomorphism group \(\operatorname{Ham}(M,\omega)\). The induced bi-invariant \textit{Hofer metric} on \(\operatorname{Ham}(M,\omega)\) is denote by $d_H$. 

For a smooth Hamiltonian
$H:S^1\times M\to\R$,  
the \textit{action functional} on $\Lambda M$ is defined as
\begin{equation}
\label{eq:action}
 \cA_H(x)=\int_{D^2}\bar x^*\omega-
 \int_0^1H_t(x(t))\,dt,
\end{equation}
where $\bar x: D^2\to M$ is a capping disk of $x$ with $\bar x|_{S^1}=x$. 
Asphericity makes the action independent of the capping disk. The \textit{action spectrum} of $H$ is defined as
\[
\mathrm{Spec}(H):= \cA_H(\mathrm{Crit}(\cA_H)).
\]
Denote by $\varphi_H^t$ the Hamiltonian flow of $H$.  
We call $H$ \textit{non-degenerate} if all contractible 1-periodic Hamiltonian orbits are non-degenerate in the sense that for each $x\in\mathrm{Crit}(\cA_H)$, $1$ is not the eigenvalue of the linearization $d\varphi_H^1:T_{x(0)}M\to 
T_{x(0)}M$. For a non-degenerate Hamiltonian $H$,   
the Floer complex $CF_*(H)$ is 
a vector space over $\mathbb{K}$ generated by contractible 1-periodic Hamiltonian orbits of $H$. 

With Schwarz's grading and Floer equation~(cf. \cite{Schwarz}), the Floer differential decreases action,
so the generators with action at most $a$ span a subcomplex
$CF_*^{(-\infty,a]}(H)$.  Write
\[
 V_a^H:=HF_*^{(-\infty,a]}(H),
 \qquad
 i_a^b:V_a^H\to V_b^H.
\]
The use of $(-\infty,a]$ fixes the endpoint convention.  Replacing it by
$(-\infty,a)$ changes only the decorations of bars and none of the estimates.

\subsection{The filtered cap-action}
Let $A=H^*(M;\K)$, with multiplication denoted by $\smilecup$.  The Floer
cap-action is denoted on the right by $\frownH$; see~\cite[Section~2.3]{Schwarz}. We give some basic properties on the filtered cap-action proved by Schwarz in~\cite{Schwarz}. 

\begin{proposition}
\label{thm:filtered-cap}
For every non-degenerate $H$, the following hold.
\begin{enumerate}
\item Each $V_a^H$ is a graded right $A$-module and every structure map
$i_a^b$ is $A$-linear.
\item On full Floer homology,
\[
 (x\frownH u)\frownH v=x\frownH(u\smilecup v).
\]
Under the PSS isomorphism $\Phi_H:H^*(M;\K)\to HF_*(H)$, with the standard
degree shift,
\begin{equation}
\label{eq:PSS-module}
 \Phi_H(\alpha)\frownH u=\Phi_H(\alpha\smilecup u).
\end{equation}
\item If $u$ has positive degree and a chain-level cap trajectory contributes
from an orbit $y$ to an orbit $z$, then
\begin{equation}
\label{eq:strict-action}
 \cA_H(z)<\cA_H(y).
\end{equation}
Consequently, if
\begin{equation}
\label{eq:hbar}
 \hbar(H):=
 \min\{\cA_H(y)-\cA_H(z)>0\mid y,z\in\mathcal P(H)\},
\end{equation}
then
\begin{equation}
\label{eq:strict-filtered-cap}
 V_a^H\frownH A^{>0}\subseteq V_{a-\hbar(H)}^H.
\end{equation}
\end{enumerate}
\end{proposition}

\noindent \textbf{Sketch of the proof.}
Following~\cite{Schwarz}, we recall the geometric argument in order to make the filtration statement
and its scope explicit.  Choose a Morse--Smale pair $(f,g)$ on $M$.  Let $p$ be a critical point of $f$ of degree $r$ which is one of generators of a
Morse cochain representing a class in $H^r(M)$.
For Floer orbits $y,z$, consider pairs $(v,\eta)$ where $v$ is a Floer
cylinder from $y$ to $z$, $\eta$ is a negative Morse half-trajectory
converging to $p$, and
\[
 v(0,0)=\eta(0).
\]
For generic data the zero-dimensional spaces define a chain map
$CF_k(H)\to CF_{k-r}(H)$.  Compactification of the one-dimensional spaces
has only Floer breaking and Morse breaking, because \eqref{eq:aspherical}
excludes sphere bubbling.  The signed boundary identity proves the chain-map
property.  The analogous moduli space with two marked constraints, together
with a parameter that changes their relative positions, proves the module
law.  These are the standard constructions in \cite{Floer} or
\cite[Section~2.3]{Schwarz}; the PSS compatibility
\eqref{eq:PSS-module} is Schwarz's Proposition~2.7, based on the PSS
identification~\cite{PSS}.

The energy identity for $v$ gives
$\cA_H(z)\le\cA_H(y)$.  If equality holds, then the energy of $v$ is zero and
$v$ is the constant cylinder over a periodic orbit.  When $r>0$, the stable
manifold used for the marked constraint has positive codimension.  After a
generic choice it avoids the finite set of evaluation points of the constant
cylinders.  Hence equality is impossible, proving \eqref{eq:strict-action}.
Since the action spectrum is finite, every strict drop is at least
$\hbar(H)$, which proves \eqref{eq:strict-filtered-cap}.  The same moduli
spaces define the operation before and after inclusion of filtered complexes,
so the structure maps are $A$-linear.  The filtered chain homotopies used for
the module law obey the same non-increasing action estimate; therefore the
module law holds already on every $V_a^H$.
\qed

\begin{remark}
If the set in \eqref{eq:hbar} is empty, then the positive-degree cap-action
is zero by \eqref{eq:strict-action}.  In the applications below, a nonzero
cap word of length $k\ge1$ automatically guarantees the existence of the
required positive gaps.
\end{remark}

\subsection{Continuation and module naturality}
For Hamiltonians $H,K$, define
\begin{equation}
\label{eq:Delta}
 \Delta^+(H,K):=
 -\int_0^1\min_M(H_t-K_t)\,dt
 =\int_0^1\max_M(K_t-H_t)\,dt.
\end{equation}
A monotone-cutoff homotopy from $H$ to $K$ defines a continuation map
\begin{equation}
\label{eq:cont-map}
 C_{H,K}^a:V_a^H\longrightarrow
 V_{a+\Delta^+(H,K)}^K.
\end{equation}

\begin{theorem}[Filtered continuation naturality]
\label{thm:continuation}
The maps \eqref{eq:cont-map} can be chosen so that they are $A$-linear on
homology.  They commute with structure maps, and the composites
$C_{K,H}C_{H,K}$ and $C_{H,K}C_{K,H}$ agree with the corresponding shifted
structure maps.  Consequently the filtered $A$-modules of $H$ and $K$ are
$\delta(H,K)$-interleaved, where
\[
 \delta(H,K)=\max\{\Delta^+(H,K),\Delta^+(K,H)\}.
\]
\end{theorem}

\begin{proof}
For the cutoff homotopy
$G_s=K+\beta(s)(H-K)$, the Floer continuation equation  has the energy
identity
\[
 0\le E(u)
 =\cA_H(x)-\cA_K(y)
 +\int_{\R}\!\int_0^1\beta'(s)
 (H_t-K_t)(u(s,t))\,dt\,ds.
\]
It follows that
\[
 \cA_K(y)\le\cA_H(x)+\Delta^+(H,K),
\]
which is precisely the filtration shift in \eqref{eq:cont-map}; see
\cite[Lemma~2.12]{Schwarz}.

To prove $A$-linearity, use the same Morse representative of a class
$u\in A$ at both ends.  Consider continuation cylinders with one marked
constraint and move the marked point from the $H$-end to the $K$-end.  The
compactification of the resulting one-dimensional parametrized moduli space
has four kinds of boundary: Floer breaking at either end and the two endpoint
positions of the marked point.  The signed boundary identity gives a chain
homotopy between
$C_{H,K}(x\frownH u)$ and $C_{H,K}(x)\frown_K u$.  Asphericity excludes
sphere bubbles.  The marked-point parameter contributes no Hamiltonian term
to the energy identity, so the homotopy has the same filtration shift
$\Delta^+(H,K)$.  Thus continuation is $A$-linear on filtered homology.

Continuation homotopy invariance identifies the composite of opposite
continuations with the continuation of the concatenated homotopy, which is
chain homotopic to the identity continuation.  After padding each one-sided
map by a structure map to the common shift $\delta(H,K)$, the composites are
the $2\delta(H,K)$ structure maps.  This is an $A$-linear
$\delta(H,K)$-interleaving.
\end{proof}

\begin{remark}[Foundational scope]
The proof above records the complete energy and compactification arguments
used later.  The standard transversality, gluing, and coherent-orientation
package for the zero- and one-dimensional moduli spaces is not redeveloped
here; it is the same package used to define Hamiltonian Floer homology, the
PSS map, and Schwarz's cap-action.
\end{remark}

\subsection{Intrinsic dependence}
A normalized Hamiltonian path defines an element of the universal cover
$\widetilde{\Ham}(M,\omega)$.  The intrinsic filtered Floer groups and their
cap-actions depend only on that element.  In the symplectically aspherical
case, Schwarz's action monodromy vanishes and the relevant Seidel action is
trivial, so the normalized filtered theory descends to the time-one map
\cite[Sections~3--4]{Schwarz}.  None of the algebraic results
below requires this descent; it is used only when formulating Hofer lower
bounds on $\Ham(M,\omega)$ rather than on its universal cover.

\section{Persistent Floer cap modules and Hofer distance}
\label{sec:floer-PCL}

Let $A=H^*(M;\K)$ and let $L=\cl_\K(M)$.  For a normalized
non-degenerate Hamiltonian $H$, set
\[
 V_a^H:=HF_*^{(-\infty,a]}(H).
\]

\begin{definition}
For $a\le b$, we define the \textit{persistent cap-length} of $H$ as
\[
 \PCL_H([a,b]):=\mcl_A\bigl(\im(i_a^b:V_a^H\to V_b^H)\bigr).
\]
For every homogeneous ideal $J\subseteq A^{>0}$, depth $\ell\ge1$, and degree
$q$, define
\[
 \Bcap_{J,\ell,q}(H):=
 \barc\bigl(a\mapsto(\CapLayer_J^\ell(V_a^H))_q\bigr).
\]
\end{definition}

\begin{theorem}
\label{thm:PCL-basic}
The function $\PCL_H$ takes values in $\{0,\ldots,L\}$ and satisfies
\[
 c\le a\le b\le d\quad\Longrightarrow\quad
 \PCL_H([a,b])\ge\PCL_H([c,d]).
\]
It vanishes below the action spectrum and equals $L$ for $[a,a]$ with $a>\sup \mathrm{Spec}(H)$.  Moreover,
\[
 \PCL_H([a,b])=
 \max\left\{\ell\ge1\ \middle|\
 \begin{array}{c}
 \text{some total-layer bar in }\bigsqcup_q\Bcap_{\ell,q}(H)\\
 \text{contains }[a,b]
 \end{array}\right\}.
\]
\end{theorem}

\begin{proof}
Apply \cref{prop:anti-recover,thm:barcode} to the filtered Floer
persistence module.  Below the smallest action value the filtered complex is
zero.  Above the largest action value the module is full Floer homology,
which is the regular $A$-module under PSS; its cap-length is $L$.
\end{proof}

For normalized Hamiltonians $H,K$, recall
\begin{align*}
 \Delta^+(H,K)&=\int_0^1\max_M(K_t-H_t)\,dt,\\
 \delta(H,K)&=\max\{\Delta^+(H,K),\Delta^+(K,H)\}.
\end{align*}

\begin{theorem}[Scalar and cap-barcode stability]
\label{thm:universal-stability}
Let \(H\) and \(K\) be normalized Hamiltonians.  Then
\begin{equation}
\label{eq:PCL-Hofer-stability}
\distE(\PCL_H,\PCL_K)
\leq
\delta(H,K)
\leq
\normHofer{H-K}.
\end{equation}

Moreover, for every homogeneous ideal
\(J\subseteq A^{>0}\), every depth \(\ell\geq1\), and every
degree \(q\),
\[
d_{\mathrm B}
\bigl(
\Bcap_{J,\ell,q}(H),
\Bcap_{J,\ell,q}(K)
\bigr)
\leq
\delta(H,K).
\]
Consequently,
\begin{equation}
\label{eq:Dcap}
\Dcap(H,K)
:=
\sup_{J,\ell,q}
d_{\mathrm B}
\bigl(
\Bcap_{J,\ell,q}(H),
\Bcap_{J,\ell,q}(K)
\bigr)
\leq
\delta(H,K)
\leq
\normHofer{H-K}.
\end{equation}

After passage to the intrinsic filtered Floer theory, let
\(\PCL_\phi\) denote the persistent cap-length function associated
with \(\phi\in\Ham(M,\omega)\).  Then
\[
\distE(\PCL_\phi,\PCL_\psi)
\leq
\distH(\phi,\psi)
\]
and
\[
\Dcap(\phi,\psi)
\leq
\distH(\phi,\psi)
\]
for all \(\phi,\psi\in\Ham(M,\omega)\).
\end{theorem}

\begin{proof}
Assume first that $H$ and $K$ are non-degenerate. 
Let
\[
\mathcal V^H
=
\bigl(V_a^H,i_{a,H}^b\bigr)_{a\leq b},
\qquad
\mathcal V^K
=
\bigl(V_a^K,i_{a,K}^b\bigr)_{a\leq b}
\]
denote the filtered Floer persistence modules of \(H\) and \(K\),
where we have temporarily added the subscripts \(H\) and \(K\) to
the structure maps in order to distinguish them.

Set
\[
\varepsilon:=\delta(H,K).
\]
By \cref{thm:continuation}, the continuation maps, padded by
structure maps from their one-sided filtration shifts to the common
shift \(\varepsilon\), define an \(A\)-linear
\(\varepsilon\)-interleaving
\[
F\colon
\mathcal V^H
\longrightarrow
T_\varepsilon\mathcal V^K,
\qquad
G\colon
\mathcal V^K
\longrightarrow
T_\varepsilon\mathcal V^H.
\]
Consequently, by \cref{thm:abstract-stability} we obtain
\[
\distE(\PCL_H,\PCL_K)
\leq
\varepsilon
=
\delta(H,K).
\]

We next compare \(\delta(H,K)\) with the Hofer norm.  Since \(H\)
and \(K\) are normalized, the function
\[
H_t-K_t
\]
has zero mean for every \(t\).  Therefore
\[
\max_M(H_t-K_t)\geq0,
\qquad
\max_M(K_t-H_t)\geq0,
\]
and hence both
\(\Delta^+(H,K)\) and \(\Delta^+(K,H)\) are nonnegative.  Moreover,
\begin{align*}
\Delta^+(H,K)+\Delta^+(K,H)
&=
\int_0^1
\left(
\max_M(K_t-H_t)+\max_M(H_t-K_t)
\right)\,dt
\\
&=
\int_0^1
\left(
\max_M(H_t-K_t)-\min_M(H_t-K_t)
\right)\,dt
\\
&=
\normHofer{H-K}.
\end{align*}
It follows that
\[
\delta(H,K)
=
\max
\left\{
\Delta^+(H,K),
\Delta^+(K,H)
\right\}
\leq
\normHofer{H-K}.
\]
This proves \eqref{eq:PCL-Hofer-stability}.

For the barcode estimate, the \(A\)-linearity of \(F_a\) and
\(G_a\) implies that, for every homogeneous ideal \(J\),
\[
F_a
\bigl(
\CapLayer_J^\ell(V_a^H)
\bigr)
\subseteq
\CapLayer_J^\ell(V_{a+\varepsilon}^K)
\]
and
\[
G_a
\bigl(
\CapLayer_J^\ell(V_a^K)
\bigr)
\subseteq
\CapLayer_J^\ell(V_{a+\varepsilon}^H).
\]
Since the continuation maps preserve the grading, they restrict
further to the degree-\(q\) components.  Thus
\[
\Psi_J^{\ell,q}(\mathcal V^H)
\quad\text{and}\quad
\Psi_J^{\ell,q}(\mathcal V^K)
\]
are ordinarily \(\varepsilon\)-interleaved.  By
\cref{thm:barcode},
\[
d_{\mathrm B}
\bigl(
\Bcap_{J,\ell,q}(H),
\Bcap_{J,\ell,q}(K)
\bigr)
\leq
\varepsilon
=
\delta(H,K).
\]
Taking the supremum over all triples \((J,\ell,q)\) proves
\eqref{eq:Dcap}.

It remains to prove the intrinsic Hofer estimates.  Let
\[
\phi,\psi\in\Ham(M,\omega).
\]
Choose a normalized Hamiltonian \(H\) generating \(\phi\), and let
\(G\) be any normalized Hamiltonian generating
\[
\phi^{-1}\psi.
\]
Define
\[
K_t(x)
:=
H_t(x)
+
G_t\bigl((\phi_H^t)^{-1}(x)\bigr).
\]
Then
\[
\phi_K^t
=
\phi_H^t\circ\phi_G^t,
\]
and consequently
\[
\phi_K^1
=
\phi\circ(\phi^{-1}\psi)
=
\psi.
\]
The Hamiltonian \(K\) is normalized, since symplectic
diffeomorphisms preserve the symplectic volume.  Furthermore,
\[
K_t-H_t
=
G_t\circ(\phi_H^t)^{-1},
\]
so composition with \((\phi_H^t)^{-1}\) preserves the maximum and
minimum, and hence
\[
\normHofer{K-H}
=
\normHofer{G}.
\]

By the intrinsic filtered Floer identification,
\[
\PCL_H=\PCL_\phi,
\qquad
\PCL_K=\PCL_\psi,
\]
and similarly the cap-decorated barcodes of \(H\) and \(K\)
represent the intrinsic barcodes of \(\phi\) and \(\psi\).
Therefore the estimates already proved give
\[
\distE(\PCL_\phi,\PCL_\psi)
\leq
\normHofer{G}
\]
and
\[
\Dcap(\phi,\psi)
\leq
\normHofer{G}.
\]
Taking the infimum over all normalized Hamiltonians \(G\)
generating \(\phi^{-1}\psi\) yields
\[
\distE(\PCL_\phi,\PCL_\psi)
\leq
\distH(\phi,\psi)
\]
and
\[
\Dcap(\phi,\psi)
\leq
\distH(\phi,\psi).
\]

For degenerate Hamiltonians, the same conclusions follow by the
intrinsic continuation-limit construction, or equivalently by
arbitrarily small non-degenerate perturbations and the estimates
above.
\end{proof}


\subsection{Symplectic push-forward and cap naturality}
\label{subsec:push-forward}

Let $\eta\in\Symp(M,\omega)$ and let $H$ be a normalized
non-degenerate Hamiltonian.  Define the conjugated Hamiltonian by
\begin{equation}
 H^\eta_t:=H_t\circ\eta^{-1}.
\label{eq:conjugated-Hamiltonian}
\end{equation}
It is again normalized and satisfies
\[
 \phi_{H^\eta}^t=\eta\phi_H^t\eta^{-1}.
\]
Push-forward sends a periodic orbit $x$ of $H$ to the periodic orbit
\[
 P(\eta)x:=\eta\circ x
\]
of $H^\eta$.  Since $(M,\omega)$ is symplectically aspherical, the action
of a contractible periodic orbit is independent of the choice of capping;
accordingly, no capped-orbit notation is needed here.  Since
$\eta^*\omega=\omega$ and $H^\eta_t(\eta(x))=H_t(x)$, one has
\[
 \cA_{H^\eta}(P(\eta)x)=\cA_H(x).
\]
Transporting the almost-complex data by $\eta$ identifies the corresponding
Floer trajectories and preserves the Conley--Zehnder grading.  Thus
push-forward induces filtration- and degree-preserving isomorphisms
\begin{equation}
 P_a(\eta):HF_*^{(-\infty,a]}(H)
 \xrightarrow{\ \cong\ }
 HF_*^{(-\infty,a]}(H^\eta),
\label{eq:push-forward-filtered}
\end{equation}
which commute with the persistence structure maps.  We use the standard
push-forward construction from
\cite[Definition~2.11]{PolterovichShelukhin}; the compatibility with
marked-point operators is also part of the operator-naturality framework in
\cite[Section~3.1]{PolterovichShelukhinStojisavljevic}.

\begin{proposition}[Push-forward naturality of cap layers]
\label{prop:push-forward-naturality}
For every $u\in H^*(M;\K)$ and every $x\in V_a^H$,
\begin{equation}
 P_a(\eta)(x\frown u)
 =P_a(\eta)(x)\frown(\eta^{-1})^*u.
\label{eq:push-forward-cap}
\end{equation}
Consequently, for every homogeneous ideal
$I\subseteq H^{>0}(M;\K)$ and every $\ell\geq0$,
\begin{equation}
 P_a(\eta)\bigl(\CapLayer_I^\ell(V_a^H)\bigr)
 =
 \CapLayer_{(\eta^{-1})^*I}^\ell(V_a^{H^\eta}).
\label{eq:push-forward-layer}
\end{equation}
The equality is compatible with the persistence structure maps and with the
gradings.  Hence, for every degree $q$,
\begin{equation}
 \Bcap_{(\eta^{-1})^*I,\ell,q}(H^\eta)
 =\Bcap_{I,\ell,q}(H).
\label{eq:push-forward-barcode}
\end{equation}
Equivalently, writing $J=(\eta^{-1})^*I$,
\begin{equation}
 \Bcap_{J,\ell,q}(H^\eta)
 =\Bcap_{\eta^*J,\ell,q}(H).
\label{eq:push-forward-barcode-target}
\end{equation}
The same statements hold for the intrinsic filtered Floer modules of the
time-one maps $\phi$ and $\eta\phi\eta^{-1}$.
\end{proposition}

\begin{proof}
Represent $\PD(u)$ by a cycle or pseudocycle
$Z_u\to M$.  In the marked-point definition of the cap-action, the relevant
Floer solutions $v:\R\times S^1\to M$ satisfy an incidence condition of the form
\[
 v(0,0)\in Z_u.
\]
Push-forward sends such a solution to $\eta\circ v$, and the incidence
condition becomes
\[
 (\eta\circ v)(0,0)\in\eta(Z_u).
\]
Naturality of Poincar\'e duality under the orientation-preserving
diffeomorphism $\eta$ gives
\[
 \eta_*\PD(u)=\PD\bigl((\eta^{-1})^*u\bigr).
\]
Thus the push-forward correspondence identifies the marked-point moduli
spaces defining the two sides of \eqref{eq:push-forward-cap}.  With the
transported coherent orientations, their algebraic counts agree, which proves
\eqref{eq:push-forward-cap} on Floer homology.

For $\ell\geq1$, iterating \eqref{eq:push-forward-cap} gives
\begin{align*}
&P_a(\eta)
 \bigl(((x\frown u_1)\frown u_2)\cdots\frown u_\ell\bigr)
\\
&\qquad=
 \bigl(((P_a(\eta)x\frown(\eta^{-1})^*u_1)
 \frown(\eta^{-1})^*u_2)\cdots\frown(\eta^{-1})^*u_\ell\bigr).
\end{align*}
For $\ell=0$, the same conclusion follows directly from
\eqref{eq:push-forward-filtered}.  This proves one inclusion in
\eqref{eq:push-forward-layer}; applying the same argument to $\eta^{-1}$
proves the reverse inclusion.  Since
\eqref{eq:push-forward-filtered} is filtration preserving, degree preserving,
and natural in $a$, it gives an isomorphism of the fixed-depth,
fixed-degree persistence modules.  Their barcodes are therefore equal, which
proves \eqref{eq:push-forward-barcode} and
\eqref{eq:push-forward-barcode-target}.  The intrinsic statement follows
from the descent of the normalized filtered theory discussed in
\cref{sec:floer}.
\end{proof}

\begin{remark}[Semilinearity in noncontractible free homotopy classes]
\label{rem:push-forward-semilinear}
Equation \eqref{eq:push-forward-cap} says that push-forward is generally
semilinear with respect to the cohomology automorphism
$(\eta^{-1})^*$; it is not $A$-linear unless this automorphism acts trivially
on the relevant classes.  The same proof applies to filtered Floer homology
in an atoroidal free homotopy class $\alpha$.  In that setting push-forward
has the form
\[
 P_a(\eta):HF_*^{(-\infty,a]}(H)_\alpha
 \longrightarrow
 HF_*^{(-\infty,a]}(H^\eta)_{\eta_*\alpha}.
\]
In particular, when $\eta_*\alpha=\alpha$, it defines an automorphism of the
filtered Floer homology in the free homotopy class $\alpha$.
\end{remark}

\begin{proposition}[Hamiltonian-conjugacy invariance]
\label{prop:conjugacy}
For $\theta\in\Ham(M,\omega)$,
\[
 \Bcap_{J,\ell,q}(\theta\phi\theta^{-1})
 =\Bcap_{J,\ell,q}(\phi)
\]
for every homogeneous ideal $J$.  Hence
\[
 \Dcap(\theta\phi\theta^{-1},\psi)=\Dcap(\phi,\psi).
\]
\end{proposition}

\begin{proof}
Apply \cref{prop:push-forward-naturality} with $\eta=\theta$.  Since a
Hamiltonian diffeomorphism is isotopic to the identity,
$\theta^*=\id$ on $H^*(M;\K)$.  Thus
\eqref{eq:push-forward-barcode-target} preserves every ideal label and gives
\[
 \Bcap_{J,\ell,q}(\theta\phi\theta^{-1})
 =\Bcap_{J,\ell,q}(\phi).
\]
The asserted equality for $\Dcap$ follows directly from its definition.
\end{proof}

\subsection{Hamiltonian-conjugacy orbits and symplectic mapping-class displacement}
\label{sec:mapping-class}

The universal cap-profile pseudodistance has a geometric consequence that
uses the ideal labels, not only the underlying collection of vector spaces.
Throughout this section, \emph{Hamiltonian conjugacy} means conjugation by an
element of $\Ham(M,\omega)$, whereas \emph{symplectic conjugation} means
conjugation by an element of $\Symp(M,\omega)$.  We write
\[
 \Csymp{\eta}{\phi}:=\eta\phi\eta^{-1},
 \qquad \eta\in\Symp(M,\omega),
\]
and, for $\phi\in\Ham(M,\omega)$,
\[
 \OrbHam(\phi):=\{\theta\phi\theta^{-1}\mid
                    \theta\in\Ham(M,\omega)\},
 \qquad [\phi]_{\mathrm{Ham}}:=\OrbHam(\phi).
\]
The Hofer pseudodistance between two Hamiltonian-conjugacy orbits is
\begin{equation}
\label{eq:dconj}
 \dHamConj([\phi]_{\mathrm{Ham}},[\psi]_{\mathrm{Ham}})
 :=\inf_{\theta_0,\theta_1\in\Ham(M,\omega)}
 \distH(\theta_0\phi\theta_0^{-1},
        \theta_1\psi\theta_1^{-1}).
\end{equation}
By bi-invariance of Hofer's metric, one may equivalently fix either one of
the two representatives and conjugate only the other.  Thus the orbit
relation in \eqref{eq:dconj} always uses Hamiltonian conjugators; a general
symplectic diffeomorphism enters only as an external map between such orbits.

\begin{theorem}[Symplectic naturality of decorated barcodes]
\label{thm:symp-naturality}
Let $\eta\in\Symp(M,\omega)$.  For every homogeneous ideal
$J\subseteq H^{>0}(M;\K)$, every $\ell\geq1$, and every degree $q$,
\begin{equation}
\label{eq:symp-bar-naturality}
 \Bcap_{J,\ell,q}(\Csymp{\eta}{\phi})
 =\Bcap_{\eta^*J,\ell,q}(\phi).
\end{equation}
Consequently,
\begin{equation}
\label{eq:mapping-profile-identity}
 \Dcap(\phi,\Csymp{\eta}{\phi})
 =\sup_{J,\ell,q}
 d_{\mathrm B}\bigl(\Bcap_{J,\ell,q}(\phi),
                     \Bcap_{\eta^*J,\ell,q}(\phi)\bigr).
\end{equation}
\end{theorem}

\begin{proof}
Equation \eqref{eq:symp-bar-naturality} is precisely the intrinsic form of
\eqref{eq:push-forward-barcode-target} in
\cref{prop:push-forward-naturality}.  Substitution into the definition of
$\Dcap$ gives \eqref{eq:mapping-profile-identity}.
\end{proof}

\begin{theorem}[Cap-profile lower bound on the Hamiltonian-conjugacy quotient]
\label{thm:conj-quotient}
For all $\phi,\psi\in\Ham(M,\omega)$,
\begin{equation}
\label{eq:conj-profile-lower}
 \dHamConj([\phi]_{\mathrm{Ham}},[\psi]_{\mathrm{Ham}})
 \geq\Dcap(\phi,\psi).
\end{equation}
In particular, for $\eta\in\Symp(M,\omega)$ define
\begin{equation}
\label{eq:DeltaCap}
 \DeltaCap_\eta(\phi)
 :=\sup_{J,\ell,q}
 d_{\mathrm B}\bigl(\Bcap_{J,\ell,q}(\phi),
                     \Bcap_{\eta^*J,\ell,q}(\phi)\bigr).
\end{equation}
Then
\begin{equation}
\label{eq:mapping-class-lower}
 \dHamConj\bigl([\phi]_{\mathrm{Ham}},
                 [\Csymp{\eta}{\phi}]_{\mathrm{Ham}}\bigr)
 \geq \DeltaCap_\eta(\phi).
\end{equation}
\end{theorem}

\begin{proof}
For $\theta_0,\theta_1\in\Ham(M,\omega)$, Hamiltonian-conjugacy invariance in
\cref{prop:conjugacy}, applied in both variables, gives
\[
 \Dcap(\theta_0\phi\theta_0^{-1},
       \theta_1\psi\theta_1^{-1})=\Dcap(\phi,\psi).
\]
The universal Hofer estimate then implies
\[
 \distH(\theta_0\phi\theta_0^{-1},
        \theta_1\psi\theta_1^{-1})\geq\Dcap(\phi,\psi).
\]
Taking the infimum proves \eqref{eq:conj-profile-lower}.  The special case
\eqref{eq:mapping-class-lower} follows from
\cref{thm:symp-naturality}.
\end{proof}

\begin{remark}[The two kinds of conjugation]
\label{rem:two-conjugations}
If $\eta\in\Ham(M,\omega)$, then
$\Csymp{\eta}{\phi}\in\OrbHam(\phi)$ and hence
\[
 \dHamConj\bigl([\phi]_{\mathrm{Ham}},
                 [\Csymp{\eta}{\phi}]_{\mathrm{Ham}}\bigr)=0.
\]
Moreover $\eta^*=\id$ on ordinary cohomology, so
$\DeltaCap_\eta(\phi)=0$.  Therefore every nonzero distance detected by
\eqref{eq:mapping-class-lower} requires an external conjugator
$\eta\in\Symp(M,\omega)\setminus\Ham(M,\omega)$, although the orbit relation
on the left is always Hamiltonian conjugacy.
\end{remark}

\begin{corollary}[Dependence on the cohomological symplectic action]
\label{cor:mapping-class-dependence}
The number $\DeltaCap_\eta(\phi)$ is unchanged if $\phi$ is replaced by a
Hamiltonian conjugate.  It is also unchanged if $\eta$ is replaced by
$\theta_0\eta\theta_1$ with $\theta_0,\theta_1\in\Ham(M,\omega)$.  More
precisely, it depends on $\eta$ only through the induced graded-ring
automorphism
\[
 \eta^*:H^*(M;\K)\longrightarrow H^*(M;\K).
\]
Consequently it is constant on symplectic isotopy classes and is determined
by the action of the symplectic mapping class on cohomology.
\end{corollary}

\begin{proof}
Hamiltonian diffeomorphisms act trivially on ordinary cohomology, and the
cap barcodes are invariant under Hamiltonian conjugation.  Both assertions
therefore follow directly from \eqref{eq:DeltaCap}.
\end{proof}

Let
\[
 \QHam(M,\omega):=\Ham(M,\omega)/\!\sim_{\mathrm{Ham\text{-}conj}}
\]
be the set of Hamiltonian-conjugacy orbits, equipped with the pseudometric
$\dHamConj$.  Let $\mathfrak B_A$ denote the product, over all homogeneous
ideals $J\subseteq A^{>0}$, depths $\ell$, and degrees $q$, of the
corresponding barcode spaces, equipped with the supremum bottleneck metric.

\begin{proposition}[The cap-profile map]
\label{prop:profile-map}
The assignment
\[
 \ProfileCap:\QHam(M,\omega)\longrightarrow\mathfrak B_A,
 \qquad
 \ProfileCap([\phi]_{\mathrm{Ham}})=
 \bigl(\Bcap_{J,\ell,q}(\phi)\bigr)_{J,\ell,q},
\]
is well defined and $1$-Lipschitz.  Symplectic conjugation defines an action
of $\Symp(M,\omega)$ on $\QHam(M,\omega)$ by
\[
 \eta\cdot[\phi]_{\mathrm{Ham}}
 :=[\Csymp{\eta}{\phi}]_{\mathrm{Ham}}.
\]
Since $\Ham(M,\omega)$ acts trivially on this quotient, the action factors
through $\Symp(M,\omega)/\Ham(M,\omega)$.  The profile map is equivariant for
the action on $\mathfrak B_A$ that permutes ideal labels by pullback:
\begin{equation}
\label{eq:profile-equivariance}
 \bigl(\ProfileCap(\eta\cdot[\phi]_{\mathrm{Ham}})\bigr)_{J,\ell,q}
 =\bigl(\ProfileCap([\phi]_{\mathrm{Ham}})\bigr)_{\eta^*J,\ell,q}.
\end{equation}
\end{proposition}

\begin{proof}
Hamiltonian-conjugacy invariance makes
\(\ProfileCap\) well defined, and the \(1\)-Lipschitz assertion is
exactly \cref{thm:conj-quotient}.

We next verify the claimed action on
\(\QHam(M,\omega)\).  Suppose that
\[
\psi=\theta\phi\theta^{-1}
\]
for some \(\theta\in\Ham(M,\omega)\), and let
\(\eta\in\Symp(M,\omega)\).  Then
\[
\eta\psi\eta^{-1}
=
(\eta\theta\eta^{-1})
(\eta\phi\eta^{-1})
(\eta\theta\eta^{-1})^{-1}.
\]
Since $\Ham(M,\omega)$ is a normal subgroup of $\Symp(M,\omega)$ (cf. \cite{HZ}), 
one has
\[
\eta\theta\eta^{-1}\in\Ham(M,\omega).
\]
Consequently,
\[
[\eta\psi\eta^{-1}]_{\mathrm{Ham}}
=
[\eta\phi\eta^{-1}]_{\mathrm{Ham}},
\]
so symplectic conjugation preserves the
Hamiltonian-conjugacy equivalence relation and therefore induces a
well-defined action of \(\Symp(M,\omega)\) on
\(\QHam(M,\omega)\).

If \(\eta\in\Ham(M,\omega)\), then
\[
[\eta\phi\eta^{-1}]_{\mathrm{Ham}}
=
[\phi]_{\mathrm{Ham}}
\]
for every \(\phi\in\Ham(M,\omega)\).  Thus
\(\Ham(M,\omega)\) acts trivially on
\(\QHam(M,\omega)\), and the action factors through
\[
\Symp(M,\omega)/\Ham(M,\omega).
\]
Finally, formula \eqref{eq:profile-equivariance} follows from
\cref{thm:symp-naturality}.
\end{proof}

There is also an asymptotic version which requires no existence theorem for a
limit.  Set
\[
 \underline{\Delta}^{\mathrm{cap}}_{\eta}(\phi)
 :=\liminf_{k\to\infty}\frac{1}{k}\DeltaCap_\eta(\phi^k).
\]
Then \cref{thm:conj-quotient} gives
\begin{equation}
\label{eq:stable-mapping-lower}
 \liminf_{k\to\infty}\frac1k
 \dHamConj\bigl([\phi^k]_{\mathrm{Ham}},
                 [\Csymp{\eta}{\phi^k}]_{\mathrm{Ham}}\bigr)
 \geq\underline{\Delta}^{\mathrm{cap}}_{\eta}(\phi).
\end{equation}
We do not claim that the right-hand side is positive in general; finding
large-scale examples is a separate geometric problem.

\section{A genuine Hamiltonian separation example}
\label{sec:real-example}

We now show that the ideal-decorated package can distinguish actual
Hamiltonians even when the ordinary Floer barcode and all ordinary spectral
invariants agree.  The Morse-theoretic input is the genus-two construction of
Polterovich--Shelukhin--Stojisavljevi\'c
\cite[Section~2.4]{PolterovichShelukhinStojisavljevic}.

\begin{lemma}[Genus-two Morse input]
\label{lem:PSSS-Morse}
There exist a closed oriented surface $\Sigma_2$, a handle-exchange
diffeomorphism $\eta:\Sigma_2\to\Sigma_2$, a Morse function
$f:\Sigma_2\to\R$, the function $g=f\circ\eta$, and a class
$e\in H_1(\Sigma_2;\K)$ such that:
\begin{enumerate}
\item the ordinary sublevel homology persistence modules of $f$ and $g$ have
identical barcodes;
\item all minimax spectral invariants agree:
$c(z,f)=c(z,g)$ for every $0\ne z\in H_*(\Sigma_2;\K)$;
\item the image persistence modules of intersection by $e$ have distinct
barcodes.  In particular, for some degree $q_0$,
\[
 c_0:=d_{\mathrm B}\bigl(\barc(\im(e\cap-)_{f,q_0}),
                          \barc(\im(e\cap-)_{g,q_0})\bigr)>0.
\]
The two image barcodes differ by a finite bar whose unscaled interval is
$(\eps,a]$ in the notation of the cited construction.
\end{enumerate}
\end{lemma}

\begin{proof}
This is precisely the construction and barcode computation in
\cite[Section~2.4]{PolterovichShelukhinStojisavljevic}.  Since the barcodes are
finite and unequal, their bottleneck distance is strictly positive.
\end{proof}

Furthermore, 
one can choose an area form $\sigma$ invariant under the handle-exchange
diffeomorphism $\eta$ such that $\eta\in\Symp(\Sigma_2,\sigma)\setminus\Ham(\Sigma_2,\sigma)$ in the above construction.  More precisely, choose the handle-exchange diffeomorphism
\(\eta:\Sigma_2\to\Sigma_2\)
to be the standard orientation-preserving involution interchanging
the two handles.  Let
\[
a_1,b_1,a_2,b_2
\]
be a standard symplectic basis of
\(H_1(\Sigma_2;\mathbb Z)\), with \(a_i,b_i\) supported on the
\(i\)-th handle.  Then, up to the harmless signs determined by the
orientations of the chosen curves,
\[
\eta_*a_1=a_2,\qquad
\eta_*b_1=b_2,\qquad
\eta_*a_2=a_1,\qquad
\eta_*b_2=b_1.
\]
In particular,
\[
\eta_*\neq\id
\quad\text{on }H_1(\Sigma_2;\mathbb Z).
\]
Since every diffeomorphism isotopic to the identity acts trivially
on homology, \(\eta\) represents a nontrivial mapping class.

To choose an invariant area form, start with any area form
\(\sigma_0\) and set
\[
\sigma
:=
\frac12\bigl(\sigma_0+\eta^*\sigma_0\bigr).
\]
Since \(\eta^2=\id\), one has
\[
\eta^*\sigma=\sigma,
\]
and hence
\[
\eta\in\Symp(\Sigma_2,\sigma).
\]
On the other hand, every Hamiltonian diffeomorphism is connected to
the identity by a Hamiltonian isotopy and therefore represents the
trivial mapping class.  Consequently,
\(
\eta\notin\Ham(\Sigma_2,\sigma),
\)
so that
\[
\eta\in
\Symp(\Sigma_2,\sigma)\setminus
\Ham(\Sigma_2,\sigma).
\]
Note that $f$ and $g$ have the same mean
value, denoted by $\bar f=\bar g$.  For sufficiently small $\tau>0$, define
normalized autonomous Hamiltonians
\begin{equation}
\label{eq:small-Hfg}
 H_f=-\tau(f-\bar f),\qquad H_g=-\tau(g-\bar g).
\end{equation}
Let $u=\PD(e)\in H^1(\Sigma_2;\K)$ and let $J_u=(u)$.

\begin{theorem}[Symplectically conjugate Hamiltonians separated in the Hamiltonian-conjugacy quotient]
\label{thm:real-Hamiltonian}
For all sufficiently small $\tau>0$, the Hamiltonians in
\eqref{eq:small-Hfg} are non-degenerate and satisfy
\begin{align*}
 \barc(HF_*^{(-\infty,\bullet]}(H_f))
 &=\barc(HF_*^{(-\infty,\bullet]}(H_g)),\\
 c(\alpha;H_f)&=c(\alpha;H_g)
 \quad(0\ne\alpha\in H^*(\Sigma_2;\K)).
\end{align*}
Moreover,
\[
 \phi_{H_g}=\Csymp{\eta^{-1}}{\phi_{H_f}},
\]
while for the principal ideal $J_u=(u)$,
\begin{equation}
\label{eq:real-bar-distance}
 d_{\mathrm B}\bigl(\Bcap_{J_u,1,q_0}(H_f),
                     \Bcap_{J_u,1,q_0}(H_g)\bigr)
 =\tau c_0>0.
\end{equation}
Consequently,
\begin{equation}
\label{eq:conjugacy-lower}
 \dHamConj\bigl([\phi_{H_f}]_{\mathrm{Ham}},[\phi_{H_g}]_{\mathrm{Ham}}\bigr)
 \geq\tau c_0.
\end{equation}
Equivalently, the two maps are conjugate by the non-Hamiltonian
symplectic diffeomorphism $\eta^{-1}$, but their Hamiltonian-conjugacy
orbits have positive Hofer distance.
\end{theorem}

\begin{proof}
For $\tau$ sufficiently small, the only contractible one-periodic orbits of
$H_f$ and $H_g$ are the critical points of $f$ and $g$.  With the action
convention of \eqref{eq:action},
\[
 \cA_{H_f}(p)=\tau(f(p)-\bar f),\qquad
 \cA_{H_g}(p)=\tau(g(p)-\bar g).
\]
The filtered Floer complexes identify with the corresponding Morse complexes
after the common affine change of filtration parameter
$t\mapsto\tau(t-\bar f)$.  Under the Morse--Floer/PSS identification,
cap-action by $u=\PD(e)$ corresponds to intersection by $e$, and
\cref{lem:principal} identifies $\CapLayer_{J_u}^1$ with the image of this
single operator.

All ordinary bars and all spectral invariants therefore undergo the same
affine rescaling for $f$ and $g$.  A common translation leaves bottleneck
distance unchanged, while multiplication of all endpoints by $\tau$ scales
it by $\tau$; this proves \eqref{eq:real-bar-distance}.

The area form was chosen $\eta$-invariant and $g=f\circ\eta$, hence
$H_g=H_f\circ\eta$.  The Hamiltonian vector fields are related by
$d\eta\,X_{H_g}=X_{H_f}\circ\eta$, so
$\phi_{H_g}=\Csymp{\eta^{-1}}{\phi_{H_f}}$.  Finally,
\cref{thm:conj-quotient} and \eqref{eq:real-bar-distance} give
\eqref{eq:conjugacy-lower}.
\end{proof}

\begin{remark}[Strictly new information on the conjugacy quotient]
\label{cor:not-classical}
The ordinary Floer barcode together with the complete family of Schwarz
spectral invariants does not determine the cap-profile pseudodistance on
Hamiltonian-conjugacy orbits.  In particular, these classical data do not
detect the positive lower bound in \eqref{eq:conjugacy-lower}.
\end{remark}

\begin{corollary}[A nontrivial symplectic-conjugation action]
\label{cor:nontrivial-mcg-action}
For the genus-two surface, the action on $\QHam(\Sigma_2,\sigma)$ induced
by symplectic conjugation is nontrivial.  More precisely, the class of the
handle-exchange map in
$\Symp(\Sigma_2,\sigma)/\Ham(\Sigma_2,\sigma)$ moves
$[\phi_{H_f}]_{\mathrm{Ham}}$ by $\dHamConj$-distance at least $\tau c_0$.
\end{corollary}

\begin{proof}
The handle exchange sends $[\phi_{H_f}]_{\mathrm{Ham}}$ to $[\phi_{H_g}]_{\mathrm{Ham}}$, and
\eqref{eq:conjugacy-lower} gives the stated positive displacement.
\end{proof}

\begin{remark}[Originality and scope]
The underlying Morse functions and the distinction by one intersection
operator are due to \cite{PolterovichShelukhinStojisavljevic}.  The new point
used here is the mapping-class interpretation: after Floer realization, the
same computation becomes a positive lower bound of $\dHamConj$-distance between Hamiltonian
conjugacy orbits of two symplectically conjugate maps.  We do not claim the
Morse barcode computation itself as new.
\end{remark}
\section{Cap-depth spectral selectors}
\label{sec:selectors}

\subsection{Definitions and minimax formula}
\label{subsec:cap-depth}
For $0\ne x\in HF_*(H)$, define  the \textit{cap-depth} of $x$ as
\begin{equation}
\label{eq:depth}
 \depth_A(x):=\max\{k\ge0\mid \CapLayer^k(x)\ne0\}.
\end{equation}
Thus every nonzero class has depth at least zero.  Let
$j_a:V_a^H\to HF_*(H)$ be the natural map and define
\begin{equation}
\label{eq:rho}
 \rho_H(x):=\inf\{a\mid x\in\im j_a\}.
\end{equation}
For $0\le k\le L$, set
\begin{equation}
\label{eq:kappa}
 \kappa_k(H):=
 \inf_{\substack{x\ne0\\\depth_A(x)\ge k}}\rho_H(x).
\end{equation}

\begin{proposition}[Persistent characterization]
\label{prop:kappa-PCL}
For $k\ge1$,
\[
 \kappa_k(H)=
 \inf\{a\mid \mcl_A(\im j_a)\ge k\}.
\]
For $k=0$,
\[
 \kappa_0(H)=\inf\{a\mid\im j_a\ne0\}.
\]
\end{proposition}

\begin{proof}
If $\mcl_A(\im j_a)\ge k$, then some class $x\in\im j_a$ has depth at
least $k$, so the right-hand side of \eqref{eq:kappa} is at most $a$.
Conversely, if $x\in\im j_a$ has depth at least $k$, then the $A$-linearity
of $j_a$ implies that the same nonzero cap word occurs in $\im j_a$.
Taking infima proves the assertion.  The case $k=0$ is identical without a
positive-degree action.
\end{proof}

\subsection{Endpoints and spectrality}
For $\alpha\in H^*(M;\mathbb{K})$, 
denote by $c(\alpha;H)$ the Schwarz's selector; see~\cite{Schwarz}.  Under PSS,
\[
 \rho_H(\Phi_H(\alpha))=c(\alpha;H).
\]

\begin{lemma}
\label{lem:regular-depth}
Under the PSS identification, the cap-depth of $\Phi_H(\alpha)$, denoted by $\depth_A(\Phi_H(\alpha))$, equals
\[
 \max\{k\ge0\mid \CapLayer^k(\alpha)\ne0\}.
\]
Write $\alpha=\sum_r\alpha_r$ with $\alpha_r\in H^r(M;\K)$.  Then
\begin{equation}
\label{eq:graded-spectral-max}
 \rho_H(\Phi_H(\alpha))=
 \max_{\alpha_r\ne0}c(\alpha_r;H).
\end{equation}
Let
\(
L=\cl_{\mathbb K}(M)
\)
be the cup-length of \(M\) over \(\mathbb K\).  Then
\[
\depth_A(\Phi_H(\alpha))\geq L
\]
if and only if the degree-zero component \(\alpha_0\) is nonzero.
\end{lemma}

\begin{proof}
The filtered Floer groups and
the maps to full Floer homology split by degree.  Therefore a sum of
homogeneous classes lies in $\im j_a$ exactly when each of its homogeneous
components lies in $\im j_a$.  This proves \eqref{eq:graded-spectral-max}.

If $\alpha_0=\lambda1$ with $\lambda\ne0$, choose
$u_1,\ldots,u_L\in\mideal$ with $u_1\smilecup\cdots\smilecup u_L\ne0$.
Every positive-degree component of $\alpha$ annihilates this length-$L$
product, since otherwise one would obtain $L+1$ positive-degree factors with
nonzero product.  Hence
\[
 \alpha\smilecup u_1\smilecup\cdots\smilecup u_L
 =\lambda u_1\smilecup\cdots\smilecup u_L\ne0.
\]
Conversely, if $\alpha_0=0$, then every homogeneous component of $\alpha$
has positive degree.  Nonvanishing of $\CapLayer^L(\alpha)$ would again
produce a nonzero product of $L+1$ positive-degree homogeneous factors after
expanding into homogeneous summands, contradicting the definition of $L$.
\end{proof}

\begin{theorem}[Selector hierarchy]
\label{thm:kappa}
For a normalized non-degenerate Hamiltonian $H$:
\begin{enumerate}
\item
\[
 \kappa_0(H)\le\kappa_1(H)\le\cdots\le\kappa_L(H).
\]
\item Every $\kappa_k(H)$ lies in $\Spec(H)$.
\item The endpoints are
\begin{equation}
\label{eq:endpoints}
 \kappa_0(H)=c([M];H),
 \qquad
 \kappa_L(H)=c(1;H).
\end{equation}
\item If $k\ge1$, then
\begin{equation}
\label{eq:gap-kappa}
 \kappa_k(H)\ge\kappa_0(H)+k\hbar(H).
\end{equation}
\end{enumerate}
\end{theorem}

\begin{proof}
The admissible set in \eqref{eq:kappa} decreases as $k$ increases, proving
monotonicity.

For a non-degenerate Hamiltonian, every spectral level $\rho_H(x)$ is an
action value, and only finitely many action values occur.  Hence the infimum
in \eqref{eq:kappa} is a minimum of a nonempty subset of $\Spec(H)$, proving
spectrality.

Schwarz's extremal inequalities (cf.~\cite[Corollary~4.10]{Schwarz}) say
\[
 c([M];H)\le c(\alpha;H)\le c(1;H)
\]
for every nonzero cohomology class $\alpha$.  Since PSS is an isomorphism,
$\kappa_0$ is the minimum of all nonzero spectral levels and therefore equals
$c([M];H)$.  By \cref{lem:regular-depth}, a class of depth at least $L$ has a
nonzero degree-zero component.  Formula \eqref{eq:graded-spectral-max} and
Schwarz's upper extremal inequality show that every such class has spectral
level $c(1;H)$: its degree-zero component is a nonzero multiple of $1$, while
all other homogeneous components have spectral level at most $c(1;H)$.
Hence $\kappa_L=c(1;H)$.

Finally, let $x$ have depth at least $k$.  Choose a nonzero cap word of
length $k$ and let $y$ be its terminal class.  The strict filtration drop in
\cref{thm:filtered-cap} gives
\[
 \rho_H(y)\le\rho_H(x)-k\hbar(H).
\]
Since $\rho_H(y)\ge\kappa_0(H)$, one has
$\rho_H(x)\ge\kappa_0(H)+k\hbar(H)$.  Taking the minimum over $x$ proves
\eqref{eq:gap-kappa}.
\end{proof}

\subsection{Continuation and widths}

\begin{theorem}[Continuation of selectors]
\label{thm:kappa-stability}
For every $k$,
\begin{equation}
\label{eq:kappa-one-sided}
 \kappa_k(K)\le\kappa_k(H)+\Delta^+(H,K).
\end{equation}
Consequently,
\[
 |\kappa_k(H)-\kappa_k(K)|\le\delta(H,K).
\]
If $r:S^1\to\R$, then
\begin{equation}
\label{eq:kappa-shift}
 \kappa_k(H+r)=\kappa_k(H)-\int_0^1r(t)\,dt.
\end{equation}
\end{theorem}

\begin{proof}
The full continuation map $C_{H,K}:HF_*(H)\to HF_*(K)$ is an $A$-linear
isomorphism, so it preserves cap-depth in both directions.  The filtered
estimate gives
\[
 \rho_K(C_{H,K}x)\le\rho_H(x)+\Delta^+(H,K).
\]
Taking the minimum over all $x$ of depth at least $k$ proves
\eqref{eq:kappa-one-sided}.  Apply the same argument in the reverse
direction for the symmetric estimate.

Adding $r(t)$ does not change the Hamiltonian vector field or the Floer
complex, while \eqref{eq:action} changes every action by
$-\int_0^1r(t)dt$.  The cap-action is unchanged, so every $\rho_H(x)$ and
hence every $\kappa_k$ changes by the same amount.
\end{proof}

\begin{definition}
Define the $k$th cap-spectral width by
\[
 w_k(H):=\kappa_k(H)-\kappa_0(H),
 \qquad 0\le k\le L.
\]
\end{definition}

\begin{corollary}
\label{cor:widths}
The widths are invariant under addition of functions of time and satisfy
\[
 0=w_0(H)\le w_1(H)\le\cdots\le w_L(H)
 =c(1;H)-c([M];H).
\]
Moreover,
\begin{equation}
\label{eq:width-Lipschitz}
 |w_k(H)-w_k(K)|
 \le\Delta^+(H,K)+\Delta^+(K,H)
 =\normHofer{H-K},
\end{equation}
and $w_k(H)\ge k\hbar(H)$ for $k\ge1$.
\end{corollary}

\begin{proof}
All statements except \eqref{eq:width-Lipschitz} follow from
\cref{thm:kappa,thm:kappa-stability}.  For the estimate, combine
\[
 \kappa_k(K)\le\kappa_k(H)+\Delta^+(H,K)
\]
with the reverse one-sided estimate for $\kappa_0$,
\[
 \kappa_0(K)\ge\kappa_0(H)-\Delta^+(K,H).
\]
This gives
$w_k(K)-w_k(H)\le\Delta^+(H,K)+\Delta^+(K,H)$.
Exchange $H,K$.  The equality with the Hofer norm of $H-K$ is immediate
from the definitions.
\end{proof}

\begin{corollary}[Action-spectrum diameter]
\label{cor:action-spectrum-diameter}
Let \(H\) be a normalized non-degenerate Hamiltonian and let
\[
L=\operatorname{cl}\bigl(H^*(M;\mathbb K)\bigr).
\]
Then
\[
\max\operatorname{Spec}(H)-\min\operatorname{Spec}(H)
\geq
c(1;H)-c([M];H)
=
w_L(H)
\geq
L\hbar(H).
\]
\end{corollary}

\begin{proof}
By spectrality,
\[
c(1;H),c([M];H)\in\operatorname{Spec}(H).
\]
Therefore
\[
\max\operatorname{Spec}(H)-\min\operatorname{Spec}(H)
\geq c(1;H)-c([M];H).
\]
The identity
\[
c(1;H)-c([M];H)=w_L(H)
\]
and the inequality
\[
w_L(H)\geq L\hbar(H)
\]
follow from \cref{cor:widths}.
\end{proof}

\subsection{Essential-bar interpretation}
For fixed $\ell,q$, the direct limit of the cap layer is
$(\CapLayer^\ell(HF_*(H)))_q$.  A barcode interval is called \textit{essential} when it is
unbounded above.

\begin{proposition}
\label{prop:essential-bars}
For $\ell\ge1$,
\[
 \kappa_\ell(H)=
 \min\{b(I)\mid I\in\bigsqcup_q\Bcap_{\ell,q}(H),\ I
 \text{ is essential}\}.
\]
For $\ell=0$, the same formula applied to the ordinary Floer barcode gives
$\kappa_0(H)$.
\end{proposition}

\begin{proof}
By \cref{lem:layer-functor},
\[
 \im\bigl(\CapLayer^\ell(V_a^H)\to \CapLayer^\ell(HF_*(H))\bigr)
 =\CapLayer^\ell(\im j_a).
\]
This space is nonzero exactly when $\mcl_A(\im j_a)\ge\ell$.  In an interval
decomposition, the map to the direct limit is nonzero exactly when an
essential bar has been born.  Apply \cref{prop:kappa-PCL}.
\end{proof}

\section{A complete Morse computation on tori}
\label{sec:torus}

Let $M=T^m$ with $m=2n$ and any translation-invariant symplectic form.  Its
cohomology algebra is
\[
 A=\Lambda_\K(e_1,\ldots,e_m),\qquad \deg e_i=1,
\]
so $L=m$.  Let $h:S^1\to[0,1]$ be a perfect Morse function with one minimum
and one maximum.  Choose weights
\[
 0<a_1\le a_2\le\cdots\le a_m
\]
and put
\[
 f(\theta_1,\ldots,\theta_m)=\sum_{i=1}^m a_i h(\theta_i),
 \qquad
 H=\eps(f-\bar f),
\]
where $\bar f$ is the mean value and $\eps>0$ is sufficiently small.
Then $H$ is non-degenerate and has no nonconstant contractible one-periodic
orbits.

For a subset $S\subseteq\{1,\ldots,m\}$, let
$e_S=\smilecup_{i\in S}e_i$.  The corresponding critical point is the maximum
in the coordinates indexed by $S$ and the minimum in the remaining
coordinates.

\begin{theorem}[Explicit torus formula]
\label{thm:torus}
For $0\le k\le m$,
\begin{equation}
\label{eq:torus-kappa}
 \kappa_k(H)=
 \eps\bar f-\eps\sum_{i=k+1}^{m}a_i
\end{equation}
provided that $ \eps>0$ is sufficiently small. 
Hence, 
\begin{equation}
\label{eq:torus-width}
 w_k(H)=\eps\sum_{i=1}^{k}a_i.
\end{equation}
Moreover, because the perfect Morse differential vanishes, all structure
maps of the filtered Floer module are injective and
\begin{equation}
\label{eq:torus-PCL}
 \PCL_H([a,b])=
 \max\left\{m-|S|\ \middle|\
 \eps\bar f-\eps\sum_{i\in S}a_i\le a\right\}
\end{equation}
for every $a\le b$, with value zero if the set is empty.
\end{theorem}

\begin{proof}
For $\eps$ sufficiently small, the Floer complex is canonically identified
with the Morse cochain complex, with action
\[
 \cA_H(p_S)=-H(p_S)
 =\eps\bar f-\eps\sum_{i\in S}a_i.
\]
The product Morse data are perfect: the Morse differential on each circle
vanishes and therefore the product differential vanishes.  Under PSS, the
class of $p_S$ is $e_S$.

The cap-depth of $e_S$ is $m-|S|$.  Indeed, multiplication by the generators
$e_i$ with $i\notin S$ gives a nonzero word of length $m-|S|$.  No longer
word is possible because every positive-degree factor raises total degree by
at least one and $A$ vanishes above degree $m$.

The condition $\depth(e_S)\ge k$ is equivalent to $|S|\le m-k$.
For a linear combination of monomials, depth at least $k$ forces at least
one monomial summand to have depth at least $k$, while its spectral level is
the maximum action among its nonzero monomial summands.  Hence the minimum
in the definition of $\kappa_k$ is attained by a single monomial.  To
minimize its action, one maximizes $\sum_{i\in S}a_i$ under the cardinality
constraint.  Since the weights are ordered, the maximum is attained by
$S=\{k+1,\ldots,m\}$, proving \eqref{eq:torus-kappa}.  Subtracting the case
$k=0$ gives \eqref{eq:torus-width}.

Finally, the filtered group $V_a^H$ is spanned by those $e_S$ whose action is
at most $a$, and the structure maps are inclusions because the differential
vanishes.  Hence $\im i_a^b=V_a^H$.  Its module cap-length is the maximum of
$m-|S|$ among its basis classes, which is \eqref{eq:torus-PCL}.
\end{proof}

\begin{remark}
The formula shows that the hierarchy is genuinely multilevel.  The endpoint
width $w_m=\eps\sum_i a_i$ is Schwarz's spectral norm, while the successive
increments recover the ordered weights $\eps a_1,\ldots,\eps a_m$.
A single spectral norm cannot retain this information.
\end{remark}

\begin{remark}[Intermediate widths are not norms]
\label{rem:intermediate-widths-not-norms}
The intermediate widths need not be non-degenerate.  For example,
on the standard two-torus let
\[
H_\delta(q,p)
=
\varepsilon
\left(
\delta h(q)+h(p)-(1+\delta)\bar h
\right),
\qquad \delta>0.
\]
Here $h:S^1\to\R$ is the perfect Morse function used as above. 
By \cref{thm:torus},
\[
w_1(H_\delta)=\varepsilon\delta.
\]
As \(\delta\to0\), the Hamiltonians converge in the Hofer norm to
\[
H_0(q,p)=\varepsilon\bigl(h(p)-\bar h\bigr).
\]
The Hofer continuity of \(w_1\) therefore gives
\[
w_1(H_0)=0.
\]
Nevertheless,
\[
\phi_{H_0}^1(q,p)
=
\bigl(q-\varepsilon h'(p),p\bigr)
\neq\id
\]
for nonconstant \(h\) and sufficiently small \(\varepsilon>0\).
Thus \(w_1\) is degenerate on
\(\Ham(T^2,\omega)\).

More generally, by allowing some of the weights in
\cref{thm:torus} to tend to zero, one obtains nontrivial
Hamiltonian diffeomorphisms for which several intermediate widths
vanish.  Hence only the endpoint
\[
w_L=c(1;\,\cdot\,)-c([M];\,\cdot\,)
\]
is asserted to be Schwarz's spectral norm; no norm property is
claimed for \(w_k\) when \(0<k<L\).
\end{remark}

\section{Transient cap modules and refined boundary depths}
\label{sec:transient}

Let
\[
 T_a^H:=\ker(j_a:V_a^H\to HF_*(H)).
\]
Because $j_a$ and all structure maps are $A$-linear, $T^H$ is a persistence
submodule in graded $A$-modules.  Its bars are precisely the finite bars of
the ordinary filtered Floer module.

For $\ell\ge0$ and degree $q$, define
\[
 \mathcal B_{\ell,q}^{\mathrm{cap,tr}}(H):=
 \barc\bigl(a\mapsto(\CapLayer^\ell(T_a^H))_q\bigr).
\]
Every bar in this barcode is finite.

\begin{definition}
Define the depth-$\ell$, degree-$q$ cap-boundary depth by
\begin{equation}
\label{eq:beta-lq}
 \beta_{\ell,q}^{\mathrm{cap}}(H):=
 \sup\{\operatorname{length}(I)\mid
 I\in\mathcal B_{\ell,q}^{\mathrm{cap,tr}}(H)\},
\end{equation}
with supremum zero for the empty barcode.  Put
\[
 \beta_\ell^{\mathrm{cap}}(H):=
 \sup_q\beta_{\ell,q}^{\mathrm{cap}}(H).
\]
\end{definition}

At depth zero, $\beta_0^{\mathrm{cap}}$ is the length of the longest finite
bar of filtered Floer homology, hence agrees with boundary depth under the
standard barcode interpretation of Floer-type complexes
\cite{Usher,UsherZhang}.

\begin{proposition}[Interval characterization]
\label{prop:beta-char}
For $\ell\ge1$,
\[
 \beta_\ell^{\mathrm{cap}}(H)=
 \sup\{b-a\mid
 \mcl_A(\im(T_a^H\to T_b^H))\ge\ell\}.
\]
For $\ell=0$, replace the condition by
$\im(T_a^H\to T_b^H)\ne0$.
\end{proposition}

\begin{proof}
By \cref{lem:layer-functor}, the transition image in the depth-$\ell$ layer
is nonzero exactly when
$\CapLayer^\ell(\im(T_a^H\to T_b^H))\ne0$, which is equivalent to module
cap-length at least $\ell$.  A transition in an interval module is nonzero
exactly when the corresponding bar contains $[a,b]$.  Taking the supremum of
$b-a$ gives the longest bar length.  The depth-zero argument is identical.
\end{proof}

\begin{theorem}[Stability of refined boundary depths]
\label{thm:beta}
For every $\ell,q$,
\[
 \left|\beta_{\ell,q}^{\mathrm{cap}}(H)
 -\beta_{\ell,q}^{\mathrm{cap}}(K)\right|
 \le2\delta(H,K),
\]
and the same estimate holds for $\beta_\ell^{\mathrm{cap}}$.
\end{theorem}

\begin{proof}
The continuation square
\[
\begin{CD}
 V_a^H @>{C_{H,K}}>> V_{a+\delta}^K\\
 @V{j_a^H}VV @VV{j_{a+\delta}^K}V\\
 HF_*(H) @>{C_{H,K}}>> HF_*(K)
\end{CD}
\]
shows that continuation sends $T_a^H$ to $T_{a+\delta}^K$.  The opposite
continuation does the reverse, and the interleaving identities restrict to
the kernels.  Since continuation is $A$-linear, every transient cap layer is
$\delta$-interleaved.  Thus the corresponding barcodes have bottleneck
distance at most $\delta$.

Under a bottleneck matching of cost $\delta$, a finite bar matched to a
finite bar changes each endpoint by at most $\delta$, so its length changes
by at most $2\delta$.  A bar matched to the diagonal has length at most
$2\delta$.  Therefore the longest finite bar lengths differ by at most
$2\delta$.  Taking the supremum over degrees preserves the estimate.
\end{proof}

\section{A cyclic extension and its present scope}
\label{sec:cyclic}

This section remains in the symplectically aspherical setting.  We fix a
nontrivial primitive free homotopy class $\alpha$ and impose the standard
$\alpha$-atoroidal condition on $\omega$ and $c_1$ as in~\cite{PolterovichShelukhin}.  The action and integer
grading are then single-valued.  Moreover, the full Floer homology in class
$\alpha$ vanishes: continuation to the zero Hamiltonian identifies it with a
complex having no one-periodic orbit in the nontrivial class.  Hence the
barcodes used below have no essential bars.

Fix a prime $p$.  Assume that $\operatorname{char}\K\ne p$, that $\K$ contains
all $p$-th roots of unity, and that for a primitive root $\zeta_p$ the equation
$x^p=\zeta_p^q$ has no solution in $\K$ whenever $p\nmid q$.  The cyclotomic
field $\mathbb Q(\zeta_p)$ is an example
\cite{PolterovichShelukhin,PolterovichShelukhinStojisavljevic}.

\subsection{The multiplicity-sensitive spread}
In this subsection we adopt the various notations and definitions from~\cite[Section~4.3]{PolterovichShelukhin}. 
 Let $(W,R)$ be a $\mathbb Z_p$-persistence module, meaning that $W$ is a persistence module with a $\Z_p$-action which is given by an automorphism $R:W\to W$
satisfying $R^p=\id$. For a $p$-th root $\zeta\neq 1$ of unity, let
$L_\zeta=\ker(R-\zeta\id)$ be its $\zeta$-eigenspace persistence module.  For
a barcode $B$ and an interval $I$, write $m(B,I)$ for the number of bars of
$B$ containing $I$, counted with multiplicity.  Define $\mu_p(B)$ as the
supremum of the $c\ge0$ for which there is an interval $I$ of length greater
than $4c$ such that
\[
 m(B,I)=m(B,I^{2c})\not\equiv0\pmod p.
\]
Set
\[
 \mu_p(W,R):=\max_{\substack{\zeta^p=1\\\zeta\ne1}}
 \mu_p\bigl(\barc(L_\zeta)\bigr).
\]
We use two algebraic facts from
\cite[Section~4.3]{PolterovichShelukhin}:
\begin{enumerate}
\item equivariantly $\eps$-interleaved $\mathbb Z_p$-modules have
$\mu_p$-values differing by at most $\eps$ (see~\cite[Proposition~4.21]{PolterovichShelukhin});
\item if $R=S^p$ for an automorphism $S$ of persistence modules, then
$\mu_p(W,R)=0$ (see~\cite[Proposition~4.19]{PolterovichShelukhin}).
\end{enumerate}

\subsection{Rotation on cap layers}
The marked-point cap construction extends to the Floer complex associated to $\alpha$-atoroidal class. 
It gives a filtration-nonincreasing right action of
$A=H^*(M;\K)$, and continuation and push-forward maps commute with this
action.  Indeed, the defining moduli spaces are the same spiked-cylinder or
pair-of-pants spaces used in the contractible case; the
$\alpha$-atoroidal condition makes the energy identity single-valued and
symplectic asphericity excludes sphere bubbling.  We use this standard
package in the form established in \cite[Section~4.4]{PolterovichShelukhin} (or see
\cite[Section~3.1]{PolterovichShelukhinStojisavljevic}).

Let $\phi\in\Ham(M,\omega)$ and consider the filtered
Floer persistence module in class $\alpha$ of the $p$-th iterate,
\[
 \mathbb V^\alpha(\phi^p)
 :=HF_*^{(-\infty,\bullet]}(\phi^p)_\alpha.
\]
Loop rotation on $\Lambda M$, induced by $t\to t+1/p$, defines an
automorphism
\[
 R_p:\mathbb V^\alpha(\phi^p)
 \longrightarrow\mathbb V^\alpha(\phi^p),
 \qquad R_p^p=\id.
\]
Equivalently, this operator is given by the push-forward  map (cf.~\cite[Definition~2.11]{PolterovichShelukhin}) induced by $\phi$ on the filtered Floer homology of $\phi^p$. As a consequence of Proposition~\ref{prop:push-forward-naturality}, we have

\begin{proposition}[Rotation preserves every cap layer]
\label{prop:rotation-cap}
For every homogeneous ideal $J\subseteq H^{>0}(M;\K)$ and every $\ell\ge0$,
\[
 R_p\bigl(\CapLayer_J^\ell(\mathbb V^\alpha(\phi^p))\bigr)
 =\CapLayer_J^\ell(\mathbb V^\alpha(\phi^p)).
\]
The same statement holds degree by degree.
\end{proposition}

For each $J,\ell,q$, let
\[
 \mathfrak C_{J,\ell,q}^{p,\alpha}(\widetilde\phi)
 :=\left(
 a\longmapsto
 \bigl(\CapLayer_J^\ell(HF_*^{(-\infty,a]}(\widetilde\phi^p)_\alpha)\bigr)_q,
 R_p\right)
\]
be the resulting $\mathbb Z_p$-persistence module.  Define
\[
 \mu_{p,J,\ell,q}^{\alpha}(\widetilde\phi)
 :=\mu_p\bigl( \mathfrak C_{J,\ell,q}^{p,\alpha}(\widetilde\phi)\bigr)
\]
and
\begin{equation}
\label{eq:mu-cap}
 \mu_p^{\mathrm{cap},\alpha}(\widetilde\phi)
 :=\sup_{\substack{J\subseteq A^{>0}\text{ homogeneous ideal}\\
                    0\le\ell\le L,\ q\in\mathbb Z}}
 \mu_{p,J,\ell,q}^{\alpha}(\widetilde\phi).
\end{equation}
At depth $\ell=0$ this includes the ordinary cyclic persistence module.  If
$J=(u)$ and $\ell=1$, \cref{lem:principal} identifies the layer with the image
of the single cap operator by $u$.

\begin{theorem}[Cyclic stability]
\label{thm:cyclic-stability}
For all $\widetilde\phi,\widetilde\psi\in\widetilde{\Ham}(M,\omega)$,
\[
 \left|\mu_p^{\mathrm{cap},\alpha}(\widetilde\phi)
       -\mu_p^{\mathrm{cap},\alpha}(\widetilde\psi)\right|
 \le p\,\widetilde d_{\mathrm H}(\widetilde\phi,\widetilde\psi).
\]
\end{theorem}

\begin{proof}
If normalized Hamiltonians $F,G$ generate the two lifts, then their $p$-th
iterates are generated by $F^{(p)}_t=pF_{pt}$ and $G^{(p)}_t=pG_{pt}$.
Continuation between these iterates has filtration shift at most
$p\normHofer{F-G}$.  The continuation maps are equivariant with respect to
loop rotation and $A$-linear with respect to the ordinary cap-action
\cite{PolterovichShelukhin,PolterovichShelukhinStojisavljevic}.  By
\cref{lem:layer-functor}, they restrict to every $J$-decorated cap layer.
Thus each pair
$ \mathfrak C_{J,\ell,q}^{p,\alpha}(\widetilde\phi)$ and
$ \mathfrak C_{J,\ell,q}^{p,\alpha}(\widetilde\psi)$ is equivariantly
$p\normHofer{F-G}$-interleaved.  The Lipschitz property of $\mu_p$ gives the
same estimate for every $J,\ell,q$.  Take the supremum and then the infimum
over generating Hamiltonians.
\end{proof}

\begin{theorem}[Distance to full powers]
\label{thm:cyclic-powers}
Let $\phi$ be the endpoint of $\widetilde\phi$.  Then
\begin{equation}
\label{eq:powers-lower}
 \distH\bigl(\phi,\Powers_p(M,\omega)\bigr)
 \ge\frac1p\,\mu_p^{\mathrm{cap},\alpha}(\widetilde\phi).
\end{equation}
\end{theorem}

\begin{proof}
Suppose first that $\psi=\theta^p$.  On the Floer persistence module of
$\psi^p=\theta^{p^2}$, push-forward by $\theta$ defines an automorphism $S$
whose $p$-th power is the rotation $R_p$ associated with $\psi$.  The same
naturality argument as in \cref{prop:rotation-cap} shows that $S$ preserves
every ideal-decorated cap layer.  Hence the restriction of $R_p$ to each cap
layer is a full $p$-th power and its multiplicity-sensitive spread vanishes.
Therefore
\[
 \mu_p^{\mathrm{cap},\alpha}(\widetilde\psi)=0
\]
for every lift relevant to the comparison with a full power.

Apply \cref{thm:cyclic-stability} to $\widetilde\phi$ and such a lift of
$\psi$:
\[
 \mu_p^{\mathrm{cap},\alpha}(\widetilde\phi)
 \le p\,\widetilde d_{\mathrm H}(\widetilde\phi,\widetilde\psi).
\]
Taking the infimum over lifts of full $p$-th powers gives
\eqref{eq:powers-lower}, exactly as in the operator argument of
\cite[Remark~3.1]{PolterovichShelukhinStojisavljevic}.
\end{proof}

\subsection{Formal strictness in the equivariant persistence category}

The depth-zero module and the depth-one principal-ideal modules are already
contained in
\cite{PolterovichShelukhin,PolterovichShelukhinStojisavljevic}.  We now show
that higher ideal layers cannot, even formally, be reconstructed from the
underlying cyclic module together with all single-operator image modules.

Let $\Sigma_g$ denote a closed oriented surface of genus $g$.  Set
\[
 M_0:=\Sigma_2\times\Sigma_3
\]
with a product area form.  This is a closed symplectically aspherical
manifold.  Choose
\[
 0\ne z\in H^1(\Sigma_2;\mathbb Q)
\]
and linearly independent classes
\[
 e_1,e_2,e_3\in H^1(\Sigma_3;\mathbb Q)
\]
spanning an isotropic subspace for the intersection form.  Via the K\"unneth
identification, the subspace
\[
 A_0:=\mathbb Q\cdot1\oplus\mathbb Q z
       \oplus E\oplus zE,
 \qquad E:=\operatorname{span}_{\mathbb Q}\{e_1,e_2,e_3\},
\]
is a graded subalgebra of $H^*(M_0;\mathbb Q)$, with
\[
 z^2=0,\qquad e_i\smilecup e_j=0,\qquad
 A_0^{>0}=\mathbb Qz\oplus E\oplus zE.
\]

\begin{theorem}[Formal strictness of higher cyclic cap layers]
\label{thm:formal-cyclic-strictness}
There exist a homogeneous ideal $J\subset A_0^{>0}$ and a constructible
$\mathbb Z_2$-persistence module $(\mathcal Q,R)$ in graded right
$A_0$-modules such that
\[
 \mu_2(\mathcal Q_q,R)=0
 \qquad\text{for every }q,
\]
and, for every homogeneous $u\in A_0^{>0}$,
\[
 \mu_2\bigl((\operatorname{im}m_u)_q,R\bigr)=0
 \qquad\text{for every }q,
\]
where $m_u(x)=x\frown u$, while
\[
 \mu_2\bigl((\CapLayer_J^2(\mathcal Q))_0,R\bigr)>0.
\]
Consequently, depth-two ideal-decorated cyclic persistence is strictly richer
than the underlying cyclic persistence module together with all principal
single-operator image modules, in the category of equivariant graded
persistence modules.
\end{theorem}

\begin{proof}
Take $J=A_0^{>0}$ and fix a nonempty bounded interval $I\subset\mathbb R$.
Let $X=Z=\mathbb Q^4$.  For
$b=b_1e_1+b_2e_2+b_3e_3\in E$, define
\[
 C_b=
 \begin{pmatrix}
 0&-b_3& b_2&0\\
 b_3&0&-b_1&0\\
 -b_2&b_1&0&0\\
 0&0&0&0
 \end{pmatrix}:X\longrightarrow Z.
\]
Define a graded right $A_0$-module $Q$ by
\[
 Q_2=X,\qquad Q_1=X\oplus Z,\qquad Q_0=Z,
\]
with all other graded pieces zero, and by the actions
\[
 x\frown z=(x,0),\qquad x\frown b=(0,C_bx),
\]
for $x\in Q_2$, and
\[
 (x,y)\frown z=-y,\qquad (x,y)\frown b=C_bx
\]
for $(x,y)\in Q_1$.  The degree-two action is consequently
$x\frown(z\smilecup b)=C_bx$.  Set
\[
 \mathcal Q:=\mathbb Q_I\otimes Q,
 \qquad R:=-\operatorname{id}_{\mathcal Q}.
\]

The graded module identities follow from
$z^2=0$, $e_i\smilecup e_j=0$, and
$b\smilecup z=-z\smilecup b$.  Moreover,
$\operatorname{rank}C_b=2$ for every $0\ne b\in E$, whereas
\[
 \sum_{b\in E}\operatorname{im}C_b
 =\operatorname{span}_{\mathbb Q}\{f_1,f_2,f_3\}\subset Z
\]
has dimension three. It follows that the ordinary degreewise bar multiplicities are
\(4,8,4\). Moreover, every homogeneous single-operator image is a direct
sum of \(0\), \(2\), or \(4\) copies of the interval module \(I\).
Hence its barcode multiplicities are all even, and therefore its
multiplicity-sensitive \(2\)-spread vanishes.  On the
other hand,
\[
 (\CapLayer_J^2(Q))_0
 =\sum_{b\in E}\operatorname{im}C_b
\]
has dimension three, so the corresponding $(-1)$-eigenspace barcode is
$I^{\oplus3}$ and has positive multiplicity-sensitive spread.  The module
relations and all rank calculations are given in
\cref{app:formal-cyclic-strictness}.
\end{proof}

\begin{remark}[Geometric scope]
\label{rem:formal-cyclic-scope}
\Cref{thm:formal-cyclic-strictness} realizes the coefficient algebra inside
the ordinary cohomology of a closed symplectically aspherical manifold, but
it does not realize $\mathcal Q$ as a filtered Floer persistence module.  In
particular, it does not prove a strictly stronger Hofer-geometric obstruction.
A concrete realization would require an equivariant marked-point Floer
calculation in which the differential, loop rotation, all single cap
operators, and the higher image sum are controlled simultaneously.  This
remains open.
\end{remark}

\section{Comparison with persistent cup products and operator persistence}
\label{sec:comparison}

For a filtration $X_t\subseteq X_s$ (which is an example of a persistent space with transition maps given by inclusions), cohomology is contravariant and
\textit{persistent cup-length} applies ring length to
$\im(H^*(X_s)\to H^*(X_t))$; see~\cite{ContessotoMemoliStefanouZhou,MemoliStefanouZhou}.  Filtered Floer homology is covariant in the
action threshold, and persistent cap-length applies module length relative to
the fixed ambient ring $H^*(M)$.  The fixed-depth flags are parallel:
\[
 H^+(X)\supseteq(H^+(X))^2\supseteq\cdots,
 \qquad
 \CapLayer^1(V)\supseteq\CapLayer^2(V)\supseteq\cdots.
\]
The explicit definition \eqref{eq:explicit-cap-layer} is essential: the
right-hand flag consists of subspaces generated by iterated cap actions, not
of an unspecified multiplication internal to $V$.

A class $u\in H^{>0}(M)$ gives a single persistence operator
$m_u(x)=x\frown u$.  By \cref{lem:principal}, its image is precisely the
first layer associated with the principal ideal $(u)$.  Thus the
ideal-decorated package contains the operator persistence of
\cite{PolterovichShelukhinStojisavljevic}; the total ideal
$J=H^{>0}(M)$ additionally records all cap words at every depth.

\begin{remark}[Representative cocycles and fixed endpoints]
If a barcode interval is written $I=\langle b(I),d(I)\rangle$, then $b(I)$
and $d(I)$ are fixed endpoints of the fixed bar.  The brackets suppress only
whether the endpoints are open or closed.  In Proposition~2.15 of
\cite{MemoliStefanouZhou}, a representative cocycle is chosen at, or just
before, the fixed death endpoint and restricted backward.  In covariant
Floer persistence, coherent cycle representatives propagate forward from the
birth endpoint.  The variable parameter is the filtration value, not the
endpoint symbol $d(I)$.
\end{remark}

\appendix
\section{Coherent representatives}
\label{app:representatives}

The notion of the \textit{persistent cup-length
diagram} of a filtration was introduced by Contessoto-M\'{e}moli-Stefanou-Zhou~\cite{ContessotoMemoliStefanouZhou} to compute the persistent cup-length invariant. 
Similarly, one can see that the persistent cap-length invariant can be retrieved from the
persistent cap-length diagram as follows. 

Let $\cV$ be a constructible pointwise finite-dimensional persistence module
in graded $A$-modules.  Choose an interval decomposition.  For every bar
$I=\langle b(I),d(I)\rangle$, choose a coherent family
$\xi_{I,t}\in V_t$ for $t\in I$ such that
\[
 v_t^s(\xi_{I,t})=\xi_{I,s}
 \quad(t\le s,\ [t,s]\subset I),
\]
and such that the active $\xi_{I,t}$ form a basis of $V_t$.

For a word $\mathbf u=(u_1,\ldots,u_\ell)$ in $A^{>0}$, define
\[
 \supp_\xi(I;\mathbf u)=
 \{t\in I\mid
 \xi_{I,t}\frown u_1\frown\cdots\frown u_\ell\ne0\}.
\]

\begin{lemma}
The support is empty or an initial subinterval of $I$.
\end{lemma}

\begin{proof}
If $s\le t$ lie in $I$, then $A$-linearity gives
\[
 v_s^t(\xi_{I,s}\frown\mathbf u)
 =\xi_{I,t}\frown\mathbf u.
\]
Nonvanishing at $t$ therefore implies nonvanishing at every earlier $s$.
\end{proof}

Define the \textit{persistent cap-length diagram} of the interval $J\subseteq \R$ as 
\[
 \dgm_\xi^{\mathrm{cap}}(J)=
 \max\{\ell\ge1\mid
 J=\supp_\xi(I;u_1,\ldots,u_\ell)
 \text{ for some }I,u_i\in A^{>0}\},
\]
with value zero if no such data exist.

\begin{theorem}
For every interval $[a,b]$,
\[
 \PCL_{\cV}([a,b])=
 \max_{J\supseteq[a,b]}\dgm_\xi^{\mathrm{cap}}(J).
\]
\end{theorem}

\begin{proof}
A support containing $[a,b]$ supplies a class surviving from $a$ to $b$ with
a nonzero cap word.  Conversely, expand a class realizing
$\PCL_{\cV}([a,b])$ in the interval basis at $b$.  At least one basis summand
has nonzero image under the same cap word, and its support contains $[a,b]$.
\end{proof}

\section{The cap-cube calculation}
\label{app:formal-cyclic-strictness}

We verify the algebraic assertions used in
\cref{thm:formal-cyclic-strictness}.  Since $\Sigma_2$ and $\Sigma_3$ are
aspherical, so is $M_0=\Sigma_2\times\Sigma_3$; in particular, every product
area form is symplectically aspherical.  Choose the classes $e_1,e_2,e_3$ in
a Lagrangian subspace of $H^1(\Sigma_3;\mathbb Q)$ for the intersection
pairing.  Then $e_i\smilecup e_j=0$ for all $i,j$, while the K\"unneth theorem
shows that the classes $z\smilecup e_i$ are linearly independent.  Hence
\[
 A_0=\mathbb Q\cdot1\oplus\mathbb Qz\oplus E\oplus zE
\]
is a graded subalgebra of $H^*(M_0;\mathbb Q)$ and, for
$J=A_0^{>0}$,
\[
 J^2=zE,\qquad J^3=0.
\]

Let $f_1,f_2,f_3,f_4$ be the standard basis of $Z=\mathbb Q^4$.  For
$b=(b_1,b_2,b_3)\ne0$, the upper-left $3\times3$ block of $C_b$ is the matrix
of the map $v\mapsto b\times v$ where $\times$ is the cross product on $\mathbb{R}^3$.  Its kernel is the line $\mathbb Qb$, so
\[
 \operatorname{rank}C_b=2.
\]
In particular,
\[
 \begin{aligned}
 \operatorname{im}C_{e_1}&=\operatorname{span}\{f_2,f_3\},\\
 \operatorname{im}C_{e_2}&=\operatorname{span}\{f_1,f_3\},\\
 \operatorname{im}C_{e_3}&=\operatorname{span}\{f_1,f_2\}.
 \end{aligned}
\]
Consequently,
\[
 \sum_{b\in E}\operatorname{im}C_b
 =\operatorname{span}\{f_1,f_2,f_3\}
\]
has dimension three.

We next check the right $A_0$-module identities.  For $x\in Q_2$ and
$b,c\in E$,
\[
 (x\frown z)\frown z=0,
 \qquad
 (x\frown b)\frown c=0,
\]
as required by $z^2=0$ and $b\smilecup c=0$.  Moreover,
\[
 (x\frown z)\frown b=C_bx,
 \qquad
 (x\frown b)\frown z=-C_bx,
\]
which agrees with
$b\smilecup z=-z\smilecup b$.  Positive-degree classes act trivially on
$Q_0$, so these identities exhaust the nonzero products.  Thus $Q$ is a
graded right $A_0$-module.

It remains to calculate all homogeneous single-operator images.  If
$u\in A_0^1$, write
\[
 u=\lambda z+b,
 \qquad \lambda\in\mathbb Q,
 \quad b\in E.
\]
Then
\[
 m_u:Q_2\longrightarrow Q_1,
 \qquad x\longmapsto(\lambda x,C_bx),
\]
has rank four when $\lambda\ne0$, rank two when $\lambda=0$ and $b\ne0$,
and rank zero when $u=0$.  Similarly,
\[
 m_u:Q_1\longrightarrow Q_0,
 \qquad (x,y)\longmapsto C_bx-\lambda y,
\]
has rank four when $\lambda\ne0$, rank two when $\lambda=0$ and $b\ne0$,
and rank zero otherwise.  If $u\in A_0^2$, then
$u=z\smilecup b$ for a unique $b\in E$, and the only possibly nonzero map is
\[
 m_u:Q_2\longrightarrow Q_0,
 \qquad x\longmapsto C_bx,
\]
which has rank zero or two.  Thus every degreewise single-operator image has
even dimension.

For $\mathcal Q=\mathbb Q_I\otimes Q$ and $R=-\operatorname{id}$, the
$(-1)$-eigenspace is the whole module.  The ordinary degreewise barcodes are
\[
 I^{\oplus4},\qquad I^{\oplus8},\qquad I^{\oplus4},
\]
and the barcode of every degreewise homogeneous single-operator image is
$I^{\oplus r}$ with $r\in\{0,2,4\}$.  Hence their multiplicity-sensitive
$2$-spreads vanish.  By contrast,
\[
 (\CapLayer_J^2(\mathcal Q))_0
 =\mathbb Q_I\otimes
   \operatorname{span}\{f_1,f_2,f_3\},
\]
whose barcode is $I^{\oplus3}$.  Choose an interval $K$ whose closure is
contained in the interior of $I$.  For every sufficiently small $c>0$, both
$K$ and $K^{2c}$ are contained in $I$, and therefore
\[
 m(I^{\oplus3},K)=m(I^{\oplus3},K^{2c})=3\not\equiv0\pmod2.
\]
It follows directly from the definition that
\[
 \mu_2\bigl((\CapLayer_J^2(\mathcal Q))_0,R\bigr)>0.
\]


\begin{thebibliography}{99}

\bibitem{BauerLesnick}
U.~Bauer and M.~Lesnick,
Induced matchings and the algebraic stability of persistence
barcodes,
\emph{J. Comput. Geom.}
\textbf{6} (2015), no.~2, 162--191.

\bibitem{ChazalEtAl}
F.~Chazal, V.~de Silva, M.~Glisse, and S.~Oudot,
\emph{The Structure and Stability of Persistence Modules},
SpringerBriefs in Mathematics, Springer, 2016.

\bibitem{ContessotoMemoliStefanouZhou}
M.~Contessoto, F.~M\'emoli, A.~Stefanou, and L.~Zhou,
\emph{Persistent cup-length},
in: 38th International Symposium on Computational Geometry (SoCG 2022),
LIPIcs 224 (2022), Art. 31, 31:1--31:17.

\bibitem{CrawleyBoevey}
W.~Crawley-Boevey,
\emph{Decomposition of pointwise finite-dimensional persistence modules},
J. Algebra Appl. \textbf{14} (2015), no. 5, 1550066.

\bibitem{MemoliStefanouZhou}
F.~M\'emoli, A.~Stefanou, and L.~Zhou,
\emph{Persistent cup product structures and related invariants},
J. Appl. Comput. Topol. \textbf{8} (2024), 93--148.

\bibitem{EntovConjugacy}
M.~Entov,
\emph{$K$-area, Hofer metric and geometry of conjugacy classes in Lie groups},
Invent. Math. \textbf{146} (2001), no.~1, 93--141.

\bibitem{Floer}  A. Floer, \emph{Symplectic fixed points and holomorphic spheres}, Comm. Math. Phys., \textbf{120}
(1989), 575--611.

\bibitem{Go} W. Gong, Quantum-Decorated Floer Persistence. In progress. 

\bibitem{Ho} H. Hofer, On the topological properties of symplectic maps. {\it Proc. Roy. Soc. Edinburgh Sect. A } {\bf 115} (1990),  25--38.

\bibitem{HZ} H. Hofer, and E. Zehnder, \textit{Symplectic invariants and Hamiltonian dynamics}. Birkh\"{a}user Advanced Texts: Basler Lehrb\"{a}cher. Birkh\"{a}user Verlag, Basel, 1994. xiv+341 pp.

\bibitem{PatelGPD}
A.~Patel,
Generalized persistence diagrams,
\emph{J. Appl. Comput. Topol.}
\textbf{1} (2018), no.~3--4, 397--419.

\bibitem{PSS}
S.~Piunikhin, D.~Salamon, and M.~Schwarz,
\emph{Symplectic Floer--Donaldson theory and quantum cohomology},
in: Contact and Symplectic Geometry (Cambridge, 1994),
Publ. Newton Inst. 8, Cambridge Univ. Press, 1996, pp. 171--200.

\bibitem{Po}  L. Polterovich, {\it The Geometry of the Group of Symplectic Diffeomorphisms}, Lectures in Mathematics ETH Z\"{u}rich, Birkh\"{a}user Verlag, Basel, 2001.

\bibitem{PolterovichShelukhin}
L.~Polterovich and E.~Shelukhin,
\emph{Autonomous Hamiltonian flows, Hofer's geometry and persistence modules},
Selecta Math. (N.S.) \textbf{22} (2016), 227--296.

\bibitem{PolterovichShelukhinStojisavljevic}
L.~Polterovich, E.~Shelukhin, and V.~Stojisavljevi\'c,
\emph{Persistence modules with operators in Morse and Floer theory},
Mosc. Math. J. \textbf{17} (2017), no. 4, 757--786.

\bibitem{Schwarz}
M.~Schwarz,
\emph{On the action spectrum for closed symplectically aspherical manifolds},
Pacific J. Math. \textbf{193} (2000), no. 2, 419--461.

\bibitem{Usher}
M.~Usher,
\emph{Boundary depth in Floer theory and its applications to Hamiltonian dynamics and coisotropic submanifolds},
Israel J. Math. \textbf{184} (2011), 1--57.

\bibitem{UsherZhang}
M.~Usher and J.~Zhang,
\emph{Persistent homology and Floer--Novikov theory},
Geom. Topol. \textbf{20} (2016), 3333--3430.

\end{thebibliography}
\end{document}